\documentclass[a4paper,reqno]{amsart}
\usepackage{hyperref}

\usepackage{amssymb,amsxtra}
\usepackage{graphicx} \usepackage{amsmath} \usepackage{amssymb}
\usepackage[svgnames]{xcolor}
\usepackage[color,matrix,arrow]{xy}
\input xy \xyoption{all}

\newtheorem{lem}{Lemma}[section]
\newtheorem{thm}[lem]{Theorem}
\newtheorem{pro}[lem]{Proposition}

\newtheorem{exa}[lem]{Example}

\newcommand{\LL}{\Lambda}

\newcommand{\tL}{\tilde{\Lambda}}
\newcommand{\olL}{\overline{\Lambda}}
\newcommand{\GG}{\Gamma}

\newcommand{\olG}{\overline{\Gamma}}

\newcommand{\cQ}{\mathcal Q}
\newcommand{\tQ}{\tilde{Q}}
\newcommand{\cE}{\mathcal E}

\newcommand{\hhom}{\mathrm{Hom}}

\newcommand{\uhom}[1]{\underline{\mathrm{Hom}}}
\newcommand{\Ext}{\mathrm{Ext}}
\newcommand{\Ker}{\mathrm{Ker}\,}
\newcommand{\Img}{\mathrm{Im}\,}
\newcommand{\add}{\mathrm{add}\,}

\newcommand{\gl}{\mathrm{gl.dim}\,}
\newcommand{\Ex}[2]{\Ext_{\Lambda}^{#1} (#2)}

\newcommand{\mmod}{\mathrm{mod}\,}
\newcommand{\Grm}{\mathrm{Gr}\,}
\newcommand{\grm}{\mathrm{gr}\,}

\newcommand{\cP}{\mathcal{P}}
\newcommand{\caC}{\mathcal{C}}

\newcommand{\arr}[2]{\begin{array}{#1}#2\end{array}}

\newcommand{\eqqc}[2]{\begin{equation}\label{#1}#2\end{equation}}

\begin{document}

\title{On $n$-translation Algebras}%\fnref{fn}}
\author[Guo]{Jin Yun Guo}
\email{gjy@hunnu.edu.cn}
\address{Jin Yun Guo\\ Department of Mathematics and Key Laboratory of HPCSIP (Ministry of Education of China), Hunan Normal  University, Changsha, CHINA}
\thanks{This work is partly supported by Natural Science Foundation of China \#10971172 and \#11271119}


%McKay quiver; covering space; Nakayama translation; finite group}

\date{}

%\fntext[fn]{This work is partly supported by Natural Science Foundation of China \#10971172 and \#11271119}
%\author[Guo]{Jin Yun Guo}
%\ead{gjy@xtu.edu.cn}
%\address[Guo]{ Department of Mathematics and Key Laboratory of HPCSIP (Ministry of Education of China), Hunan Normal  University, Changsha, CHINA}

%\subjclass{Primary {16G20, 20C05}; Secondary{ 16S34,16P90}}

%\keywords{}%$n$-translation quiver; $n$-translation algebra; $(p,q)-Koszul algebra; $n$-almost split sequence, periodic algebra}

%\date{}

\begin{abstract}
Motivated by Iyama's higher representation theory,  we introduce $n$-translation quivers and $n$-translation algebras.
The classical $\mathbb Z Q$ construction of the translation quiver is generalized to construct an $(n+1)$-translation quiver from an $n$-translation quiver, using trivial extension and smash product.
We find a sufficient and necessary condition for the quadratic dual of $n$-translation algebras to have $n$-almost split sequences in the category of its projective modules.
\end{abstract}

\maketitle

%\tableofcontents

\section{Introduction}
Auslander  Reiten quiver play a very important role in the representation theory of algebras\cite{ars, r1, ass} and it is widely used in the study of Cohen-Maucauley modules, cluster algebras and Calabi-Yau algebras and categories\cite{a,birs,k1,r,gls}.
There are many algebras, such as Auslander algebras, preprojective algebras,   related to Auslander Reiten quivers\cite{ars, dr, bgl}.
One important feature of Auslander Reiten quiver is that it is a translation quiver in the sense that there is a translation, related to the almost split sequences, equipped with it.
Recently, Iyama has developed the higher representation theory \cite{i1, i2, i4}, where a class of higher translation quivers also play an important role.

In  \cite{g02}, we introduce translation algebras as algebras with translation quivers as their quivers and the translation corresponds to an operation related to the  Nakayama functor.
They are algebras close to self-injective algebras with the vanishing radical cube.
The quadratic dual of Auslander algebras and of preprojective algebras are translation algebras.
In this paper, we introduce $n$-translation quivers and $n$-translation algebras as a generalization of translation quivers and translation algebras.
Besides the $n$-translation quiver, we also need $(p,q)$-Koszulity which unifies Koszulity in \cite{bgs} and almost Koszulity in \cite{bbk} in the definition of an $n$-translation algebra.

Another motivation of this paper is Iyama's construction of an $(n+1)$-complete algebra from the cone of an $n$-complete algebra, the bound quivers of these algebras are higher translation quivers \cite{i4}.
It is observed that the  quiver of an $(n+1)$-complete algebra is certain truncations of copies of the quiver of the $n$-complete algebra connected by new arrows \cite{g14}.
It is well known that the preinjective and preprojective components of the Auslander Reiten quiver are also obtained by connecting copies of the quiver (or opposite quiver) of the original algebra by certain new arrows \cite{r1}.
Iyama's construction of  absolute $n$-complete algebras is realized as truncations of the McKay quivers of Abelian groups in \cite{g13}, where the new arrows are obtained by inserting 'returning arrows' in the original quiver.

Returning arrows also appear in recent research on cluster algebras \cite{gls}.
In \cite{g14}, such returning arrows are obtained by constructing a trivial extension for a graded self-injective algebra.
Trivial extension is used in this paper to construct $(n+1)$-translation quiver from admissible $n$-translation quiver.
For an $n$-translation algebra whose the trivial extension is  $(n+1)$-translation algebra, which we called extendable, we define its $v$-extension as the smash product of its trivial extension with $\mathbb Z_v$, for a non-negative integer $v$.
The quiver of a  $v$-extension of an $n$-translation algebra with quiver $Q$ is $\mathbb Z_v|_n Q$, a construction generalizes the stable translation quiver  $\mathbb Z Q$ in the classical cases.
We prove that in stable Koszul case, the trivial extension of an $n$-translation algebra is an $(n+1)$-translation algebra.

Almost Koszul rings are introduced by Brenner, Butler and King in \cite{bbk}.
The almost Koszulity builds a bridge between $n$-translation algebras and  $n$-almost split sequences.
We find a sufficient and necessary condition for  the category of finitely generated projective modules of the quadratic dual of $n$-translation algebras to have $n$-almost split sequences in the category of its projective modules.
We also characterize the $n$-almost split sequences using $\tau $-hammocks of the $n$-translation quiver.
In fact, the Koszul complexes provide the correspondence between radical layers of the indecomposable projective module of the $n$-translation algebra and terms in the $n$-almost split sequences starting from the corresponding indecomposable projective module of its quadratic dual, which we call a partial Artin-Schelter $n$-regular algebra.

As examples, we show how our theory is applied to in the path algebras of Dynkin type, and to Iyama's absolute $n$-complete algebra $T_{m}^{n}(k)$ for $m=4$ and $n=0,1$.
In a forthcoming paper \cite{gl15}, we study certain classes of $n$-translation algebras which include the quadratic dual of Iyama's absolute $n$-complete algebra $T_{m}^{n}(k)$.
From the viewpoint of algebras, our $n$-translation algebras look like a quadratic dual version of algebras in $n$-representation theory \cite{i1, i2, i4, io, hio}.

In Section 2, we recall concepts and results on algebras, quivers, and higher representation theory that we are needed.
Notions and results on almost Koszul algebra are recalled and modified in the last part of this section.
$n$-translation quivers and $n$-translation algebras are introduced in Section 3.
In Section 4, we describe the trivial extension of algebras given by a admissible $n$-translation quivers and  prove that the trivial extension of a stable Koszul $n$-translation algebra is $(n+1)$-translation algebra.
The $(n+1)$-translation quiver of a $v$-extension of an extendable $n$-translation algebra is described in Section 5.
In Section 6, we introduce partial Artin-Schelter $n$-regular algebra and verify the Koszul duality between $n$-translation algebras and partial Artin-Schelter $n$-regular algebras.
In section 7,  $n$-almost split sequences for in the category of the finite generated projective modules of a  partial Artin-Schelter $n$-regular algebra is discussed,  using Koszul complexes.
Some examples are given in the last section to indicate how our theory are related to the representation theory of path algebras of Dynkin quiver.

\section{Preliminaries}
\subsection*{Algebras and quivers}
Let $k$ be a field and let $\Lambda$ be a graded algebra, that is, $\Lambda = \Lambda_0 + \Lambda_1 + \cdots $, with $\Lambda_0$ semisimple and $\Lambda$ is generated by $\Lambda_1$ over $\Lambda_0$.
We assume that $\Lambda$ is {\em locally finite} in the sense that $\Lambda_0$ is a direct sum (possibly, infinite copies) of $k$ and for any element $x \in \Lambda$, there is an idempotent $e \in \Lambda$ such that $x \in e \Lambda e$, and $e \Lambda_1 e$ is finite-dimensional.
If $e \Lambda e$ is finite-dimensional for each idempotent $e$, $\Lambda$ is called {\em locally finite-dimensional}.
Write $J=\Lambda_1 + \Lambda_2 + \cdots $, call it the Jacobson radical of $\Lambda$.

Let $\Lambda_0 = \bigoplus_{i\in Q_0} k_i$, with $k_i \simeq k$ as algebras.
Write $\otimes = \otimes_{\Lambda_0 } $, the tensor product over $\Lambda_0$, set $M^{ \otimes 1} =M$ and $M^{ \otimes t+1} =M^{ \otimes t}\otimes M$ for a $\Lambda_0$-$\Lambda_0$-bimodule $M$.
Let $T_{\Lambda_0 }(\Lambda_1)= \bigoplus_{t=0 }^{ \infty }\Lambda_1^{ \otimes t}$ be the tensor algebra of $\Lambda_1$ over $\Lambda_0$, with multiplication given by the tensor product.
Then $\Lambda_t$ are $\Lambda_0$-$\Lambda_0$-bimodules and we have compatible bimodule epimorphisms $\mu_t: \Lambda_1^{ \otimes t} \longrightarrow \Lambda_t$.
This induces and epimorphism of algebras $$\mu: T_{ \Lambda_0 }(\Lambda_1) \longrightarrow \Lambda$$ with $\Ker \mu \subseteq \bigoplus_{t=2 }^{ \infty }\Lambda_1^{ \otimes t}.$

Such algebras are presented by quivers with relations.
Let $e_i$ be the image of the identity of $k$ under the canonical embedding of the $k_i$ into $\Lambda_0$.
Then $\{e_i| i \in Q_0\}$ is a complete set of orthogonal primitive idempotents in $\Lambda_0$ and $$\Lambda_1 = \bigoplus_{i,j \in Q_0 }e_j \Lambda_1 e_i$$ as $\Lambda_0$-$\Lambda_0$-bimodules.
Fix a basis $Q_1^{ij}$ of $e_j \Lambda_1 e_i$ for any pair $i, j\in Q_0$, the elements of $Q_1^{ij}$ are called arrows from $i$ to $j$.
Take $$Q_1= \cup_{(i,j)\in Q_0\times Q_0} Q_1^{ ij}.$$
Thus we get a quiver $Q$ with vertex set $Q_0$ and arrow set $Q_1$ which is {\em locally finite} in the sense that there are only finitely many arrows starting or ending at each vertex.
The path algebra $kQ$ of $Q$ is defined as usual, and we have that $$kQ \simeq T_{ \Lambda_0 }(\Lambda_1).$$
For an arrow $\alpha: i \to j$ in the quiver, denote by $s(\alpha)=i$  its starting vertex and by $t(\alpha)=j$ its ending vertex.
For a path $p=\alpha_r\cdots \alpha_1$, denote by $s(p)= s(\alpha_1)$  its starting vertex and by $t(p)= t(\alpha_r)$ its ending vertex.
For a path $p$, call the number of arrows in $p$ the length of $p$ and write it as $l(p)$.

Thus $\Lambda  \simeq  kQ/I$ for some ideal $I= \Ker \mu \subseteq \bigoplus_{t \ge 2} \Lambda_1^{\otimes t}$, the ideal generated by the paths of length larger than or equal to $2$ in $kQ$.
The grading of $\Lambda$ is induced by the lengths of the paths, $\{e_i| i\in Q_0\}$ is a complete set of the paths of length zero, which forms a basis of $\Lambda_0$.
Let $\rho $ be a set of generators of $I$ and call it a {\em relation set}.
The elements in $\rho $ are linear combinations of the paths of lengths larger than or equal to two in $kQ$.
From now on, by a quiver we usually mean a {\em bound quiver} $Q = (Q_0, Q_1, \rho)$, that is, a quiver with  the vertex set $Q_0$, the arrow set $Q_1$  and the relation set $\rho$.We may assume that $\rho$ is a set of linear combinations of the paths of the same length, since $\Lambda $ is graded.
$\Lambda $ is called {\em quadratic} if $I$ is generated by $I \cap (\Lambda_1\otimes \Lambda_1)$.
In this case, $\rho$ can be chosen as a subset of $\Lambda_1\otimes \Lambda_1$.

We assume that the algebras are connected in the sense its quiver is connected throughout this paper, when not specialized.
We use the same notation $Q$ for both the quiver and the bound quiver, denote by $ k(Q)= kQ/(\rho)$ the algebra given by the bound quiver $Q$, that is, the quotient algebra of the path algebra $kQ$ of $Q$ modulo the ideal generated by the relation set $\rho$.
A path is called a {\em  bound path} if its image in $k(Q)$ is nonzero.
Similarly, we call an element $u =\sum_p c_p p$ a bound element if its image in $k(Q)$ is nonzero, and call $p$ a path of $u$ if $c_p\neq 0$.

For a locally finite graded algebra $\Lambda$ over a field $k$, denote by $\mathrm{Mod}\, \Lambda$  the category of left $\Lambda $-modules, morphisms and extensions in this category will be denoted by $\hhom_{\Lambda}$ and $ \Ext_{\Lambda}$.
Denote by $\Grm \Lambda$  the category of graded left $\Lambda$-modules, morphisms and extensions in this category will be denoted by $\mathrm{hom}_{\Lambda}$ and $ \mathrm{ext}_{ \Lambda}$.
A module $M$ in $\Grm \Lambda$ has a decomposition $M = \sum_{ t\in \mathbb Z} M_t$ as a direct sum of vector spaces satisfying $\Lambda_tM_s \subseteq M_{t+s}$ for $t,s\in \mathbb Z$.
For $M, N$ in $\Grm \Lambda$, $\mathrm{hom}_{\Lambda}$ is the set of homomorphisms $f$ satisfying $f(M_t)\subseteq N_t$ for each $t\in \mathbb Z$.
Write $[\,\, ]$ for the shifting functor, that is $M[r]_t = M_{t+r}$ for any graded $\Lambda$-module $M$.
We have that $\hhom_{\Lambda}(M, N) = \sum_{t \in \mathbb Z} \mathrm{hom}_{\Lambda} (M, N[t])$ for finitely generated graded $\Lambda$-module $M$ and $N$.
In this case, both $\hhom_{\Lambda}(M, N)$ and $ \Ext_{ \Lambda }(M, N)$ are graded naturally as  right $\LL$-modules.

A graded $\Lambda$-module  $M = \sum_{t} M_t$ is called {\em locally finitely generated (resp. locally finite-dimensional)} if for any idempotent $e \in \Lambda$, the submodule $\Lambda e M$ is finitely generated (resp. finite-dimensional).
If $\Lambda$ is locally finitely generated graded algebra and $M, N$ are locally finitely generated, then $\hhom_{\Lambda}(M, N)$ and $ \Ext_{ \Lambda }(M, N)$ are locally finitely generated  right $\LL$-modules.

Denote by $\grm \Lambda$ the category of locally finitely generated graded left $\Lambda$-modules.
If $\Lambda $ is  locally finitely generated, $\grm \Lambda$ is an abelian  category closed under extensions.
Left graded $\Lambda$-modules are also left graded $\Lambda_0 $-modules.
Denote by $\mmod\Lambda$ the category of locally finitely generated left $\Lambda$-modules.

Let $M$ be a locally finitely generated $\Lambda$-module.
Clearly, $M = \sum_{i\in Q_0} e_i M$,  each $\Lambda e_i M$ is finitely generated and we have that $$M =\bigoplus_t\bigoplus_{i\in Q_0} e_i M_t$$ as directed sum of  vector spaces and each  $e_i M_t$ is finite-dimensional.
Define the local dual $$ D_* M =\bigoplus_t\bigoplus_{i\in Q_0} \hhom_k ( e_i M_t, k) \subseteq \hhom_k (M, k)=D M, $$ as vector spaces.
This defined a locally finitely generated right $ \Lambda$-module via the action $$(\sum_{i\in Q_0} f_i)\alpha  (\sum_{j\in Q_0} x_j) = \sum_{i\in Q_0} f_i(\alpha  x_i)$$ for any $\alpha \in \Lambda$, $f =\sum_{i\in Q_0} f_i \in D_* M$ and $x =\sum_{j\in Q_0} x_j \in M$.
In fact, it is a duality between the categories of locally finitely generated left $\Lambda$-modules and the categories of locally finitely generated right $\Lambda $-modules, which restricts to the usual dual functor on the finite-dimensional modules.

Write $P(i) = \Lambda e_i$ for the indecomposable projective module corresponding to the vertex $i$.
Then $ S(i) \simeq  P(i)/J P(i)$ is the top of $P(i)$.
Let $I(i)$ be the injective envelope of $S(i)$.

For a $\Lambda$-module $M$, write $P(M)=P^{0} (M)$ for its projective cover, and $$\cdots\longrightarrow P^{r} (M) \stackrel{f_r}{\longrightarrow} \cdots \stackrel{f_2}{\longrightarrow} P^1(M) \stackrel{ f_1}{\longrightarrow} P^0(M) \stackrel{f_0}{\longrightarrow} M \longrightarrow 0$$ for a minimal projective resolution of $M$.
Recall for a $t \ge 0$, the  syzygy $\Omega^t M$ of $M$ is defined as $\Omega^{t} M =\Img f_{t}$.

\subsection*{Higher representation theory}
Iyama recently introduces higher Auslander-Reiten theory, we recall some basic notions \cite{i4}.
For $M\in\mmod\Lambda$, denote by $\add M$ the subcategory of $\mmod\Lambda$ consisting of direct summands of finite direct sums of copies of $M$.

A subcategory $\caC$ of  $\mmod\Lambda$ is called {\em contravariantly finite} if for any $X \in \mmod\Lambda$, there exists a morphism $f \in\hhom_{\Lambda}(C,X)$ with $C \in \caC$ such that $\hhom_{\Lambda} (-,C) \stackrel{f}{\longrightarrow} \hhom_{\Lambda} (-,X) \longrightarrow 0$ is exact on $\caC$.
Dually a {\em covariantly finite subcategory} is defined.
A contravariantly and covariantly finite subcategory is called {\em functorially finite}.

Let $\caC$ be a Krull-Schmidt category.
\begin{itemize}
\item[(a)] For an object $M\in\caC$, a morphism $f_0\in J_{\caC}(M,C_1)$ is called {\em left almost split} if $C_1\in\caC$ and \[\hhom_{\caC} (C_1,-) \stackrel{f_0} {\longrightarrow} J_{\caC}(M,-) \longrightarrow0\] is exact on $\caC$. A left minimal and left almost split morphism is called a {\em  source morphism}.

\item[(b)] We call a complex
    \eqqc{Eqqa}{ M\stackrel{f_0}{ \longrightarrow} C_1 \stackrel{f_1}{\longrightarrow} C_2 \stackrel{f_2}{\longrightarrow} \cdots
    }
    a {\em source sequence} of $M$ if the following conditions are satisfied.
 \begin{itemize}
 \item[(i)] $C_i\in\caC$ and $f_i\in J_{\caC}$ for any $i$,

 \item[(ii)] we have the following exact sequence on $\caC$.
     \eqqc{Eqqb}{ \cdots \stackrel{f_2}{\longrightarrow} \hhom_{\caC}(C_2,-)\stackrel{f_1}{\longrightarrow}\hhom_{\caC}(C_1,-)\stackrel{f_0}{\longrightarrow}J_{\caC}(M,-)\longrightarrow0
     }
\end{itemize}
A {\em sink morphism} and a {\em sink sequence} are defined dually.

\item[(c)] We call a complex
\[0\longrightarrow M \stackrel{}{\longrightarrow} C_1 \stackrel{}{\longrightarrow} C_2 \stackrel{}{\longrightarrow} \cdots \stackrel{}{\longrightarrow} C_{n} \stackrel{}{\longrightarrow} N \longrightarrow 0 \]
an {\em $n$-almost split sequence} if this is a source sequence of $M\in\caC$ and a sink sequence of $N\in\caC$.
\end{itemize}

A source sequence \eqref{Eqqa} corresponds to a minimal projective
resolution \eqref{Eqqb} of a functor $J_{\caC}(M,-)$ on $\caC$.
Thus any indecomposable object $M\in\caC$ has a unique source sequence
up to isomorphisms of complexes if it exists.

\subsection*{$(p,q)$-Koszulity}

Let $\Lambda$ be a locally finite graded algebra.
Assume that $\{S(i) = \Lambda_0 e_i| i\in Q_0\}$ is a complete set of non-isomorphic graded  $\Lambda_0 $-modules concentrated at degree $0$, which is also  a complete set of left graded simple $\Lambda$-modules.
In this section, we discuss almost Koszul (or $(p,q)$-Koszul) algebras, following \cite{bbk} and \cite{bgs}.

With the convention that $\Lambda_{\infty}=0$ and $M_{\infty}=0$ for any left graded $\Lambda$-module $M$, we have the following version of {\em left almost Koszul ring}  in \cite{bbk}.

For $p, q \in \mathbb N \cup \{\infty \}$, a locally finitely generated graded algebra $\Lambda = k(Q)$ is called {\em left almost Koszul} of type $(p,q)$ if it satisfies the following conditions:
\begin{enumerate}
\item[(1)] $\Lambda_t=0$ for $t > p$, and for each $i\in Q_0$,

\item[(2)] there is a graded complex

\eqqc{def_pqK}{ P^{\bullet}(i): 0 \longrightarrow P^q(i) \longrightarrow \cdots \longrightarrow P^1(i) \longrightarrow P^0(i) \longrightarrow 0}
of projective left $\Lambda$-modules such that each $P^t(i)$ is generated by its component of degree $t$ with non-zero homology $S(i)=\Lambda_0 e_i$ in degree $0$, the possible only another non-zero homology in degree $p+q$,  and in this case, it is $W(i) = \Lambda_p \otimes P^q_q(i)$.
\end{enumerate}
A simple module $S(i)$ satisfying (2) in above definition will be called a {\em  $(p,q)$-Koszul simple module.}
Right almost Koszul algebra is defined similarly using right modules.
In this paper, such algebra is abbreviated as {\em left (right)  $(p,q)$-Koszul algebra}.

Obviously, a Koszul algebra of Loewy length $p+1$ is a $(p,q)$-Koszul for $q > \gl \Lambda$.
A $(p,q)$-Koszul algebra $\Lambda$ is Koszul if one of $p, q$ is $\infty$.

Similar to \cite{bbk}, we also have the following results on the left (right) $(p,q)$-Koszulity for locally finitely generated graded algebras, and we omit the proofs since they can be easily done by  modifying those in \cite{bbk}.

\begin{pro}\label{char_pqk}
Assume $q\ge 1$, then $\Lambda $ is left $(p,q)$-Koszul if and only if

\begin{enumerate}

\item $\Ext_{\Lambda}^t (S(i), S(j))_s =0$ for $i,j \in Q_0$ and $s \neq t \le q$, and

\item $\Ext_{\Lambda}^{q+1} (S(i), S(j))_s =0$ for $s \neq p+q$.
\end{enumerate}
\end{pro}

\begin{pro}\label{lrpq} A left  $(p, q)$-Koszul algebra is also a right $(p, q)$-Koszul algebra.
\end{pro}

Thus we do not need to distinguish left and right for  $(p,q)$-Koszul algebra, we just call it a $(p,q)$-Koszul algebra.

\begin{pro}\label{bbk1} A $(p, q)$-Koszul algebra $\Lambda$ is generated by $\Lambda_0$ and $\Lambda_1$, and, if $q \ge 2$, then $\Lambda$ is quadratic.
\end{pro}

Now we consider quadratic algebra $\Lambda$.
Take the arrows from $i$ to $j$ in $Q_1$ as a basis of $ e_j \LL_1 e_i$.
Let $R$ be the $\LL_0$-$\LL_0$-sub-bimodule of $\LL_1\otimes \LL_1$ generated by the elements in $\rho$.

Recall that for two graded $\LL_0$-$\LL_0$-modules $M = \bigoplus M_t$ and  $N = \bigoplus N_t$, the gradation of the tensor product is defined by $(M\otimes N)_t =\bigoplus_{t'+t"=t} M_{t'}\otimes N_{t"}$.
As in \cite{bbk,bgs}, we define $\LL_0$-$\LL_0$-sub-bimodules $K^t$ of $\LL_1^{\otimes t}$ by the formulae $K^0 = \LL_0$, $K^1 = \LL_1$, $K^2 = R$ and $K^{t+1} = \LL_1 K^t\cap K^t\LL_1$ for all $t \ge 2$.
$K^t$ is a $\LL_0$-$\LL_0$-bimodule concentrated in degree $t$.
The left Koszul complex of $\Lambda$ is
\eqqc{KoszulComplex}{\cdots \longrightarrow \Lambda \otimes K^t \longrightarrow \cdots \longrightarrow \Lambda \otimes K^1\longrightarrow \Lambda \otimes K^0 \longrightarrow 0. }
The differential is the restriction of the map $\Lambda \otimes \LL_1^{\otimes t} \longrightarrow \Lambda \otimes \LL_1^{\otimes t-1}$, $$a\otimes v_1  \otimes \cdots \otimes v_t \longrightarrow a v_1  \otimes \cdots \otimes v_t,$$ which is the  composite of the following sequence of natural maps
$$ \Lambda \otimes K^t \hookrightarrow \Lambda \otimes \LL_1 K^{t-1} \longrightarrow \Lambda \Lambda_1 \otimes K^{t-1}\hookrightarrow \Lambda \otimes K^{t-1}.$$

Similar to Theorem 2.6.1 in \cite{bgs}, $ \Lambda$ is a Koszul if and only if its left Koszul complex is a projective resolution of $\Lambda_0$.
For each $i \in Q_0$, \eqref{KoszulComplex} restrict to a left Koszul complex for $\Lambda_0 e_i$:
\eqqc{KoszulComplexi}{\cdots \longrightarrow \Lambda \otimes K^t e_i \longrightarrow \cdots \longrightarrow \Lambda \otimes K^1 e_i \longrightarrow \Lambda \otimes K^0 e_i\longrightarrow 0. }
We can also decompose each term as a direct sum of indecomposable projective modules:
$  \Lambda  \otimes K^t e_i =  \bigoplus_{j\in Q_0}\Lambda e_j \otimes e_j  K^t e_i $.
The following proposition is an adaptation of Proposition 3.9 of \cite{bbk}.

\begin{pro}\label{pqkoszul} Assume   $p, q \ge 2$,
let $\Lambda$  be a quadratic algebra. Then $\Lambda$ is $(p, q)$-Koszul  if and only if

\begin{enumerate}
\item $\Lambda_t =0$ for all $t > q$,
\item $K^s =0 $ for all $s > p$, and
\item for each $i\in Q_0$, the only non-zero homology modules of \eqref{KoszulComplexi} are $\Lambda_0 e_i$ in degree $0$, and $0$ or $\Lambda_p \otimes K^q e_i$ in degree $p + q$.
\end{enumerate}
\end{pro}
In this case the left Koszul complex for $\Lambda$ provides the first $q+1$ terms of the minimal projective resolution of $\Lambda_0$.
There is another description of the left and right Koszul complexes associated with the quadratic algebra $\Lambda$ in 2.8 of \cite{bgs}, which makes use of the quadratic dual of $\Lambda$.

There is a decomposition of $\LL_1 =\bigoplus_{i,j} e_j \LL_1 e_i$ as $\LL_0$-$\LL_0$-bimodule, and we have that $D e_j \LL_1 e_i \simeq e_i \LL_1 e_j$ as $\LL_0$-$\LL_0$-bimodule.
Thus $D_* \LL_1$ defines the opposite quiver $Q^{op}$ of $Q$, with $Q_0^{op} =Q_0$ and $Q_1^{op}=\cup_{(i,j)\in Q_0\times Q_0} Q_1^{ *ij}$, with $Q_1^{ *ij} =\{\alpha^*|\alpha\in Q_1^{ ji}\} $ whose elements form a dual basis of the arrows in $Q_1^{ ji}$.
Note that $(D_* \LL_1 \otimes \cdots \otimes D_*\LL_1)$ is identified with $D_* (\LL_1 \otimes \cdots \otimes \LL_1)$, via $$f_t\otimes \cdots \otimes f_1(x_1\otimes \cdots \otimes x_t) = f_t( \cdots f_2(f_1(x_1) x_2) \cdots x_t) =f_1(x_1) f_2(x_2) \cdots  f_t(x_t),$$ for $f_1,\ldots,  f_t \in D_* \LL_1$ and $x_1,\ldots,  x_t \in \LL_1$.
The  quadratic dual algebra of $\Lambda$, denoted $\Lambda^!$, is defined by the formula $$\Lambda^! = T_{\LL_0} (D_*\LL_1) /(R^{\perp}),$$
where $R^{\perp} =\{f \in D_*(\LL_1^2)| f(x)=0 , \mbox{ for any } x\in R\} \subset D_*(\LL_1^2) = (D_*\LL_1)^2$ is the annihilator of $R \subset \LL_1^2$.
The following proposition follows directly from the definition.
\begin{pro} If $\Lambda$ is locally finite, so is $\Lambda^!$ and
$$ (\Lambda^!)^! = \Lambda.$$
\end{pro}

Thus the bound quiver of $\Lambda^!$ is the opposite quiver $Q^{op}= (Q_0, Q_1^{op}, \rho^*)$ of the bound quiver $Q$ of $\Lambda$ with the dual basis $\rho^*$ of a chosen base $\rho$ of $R$ in $(D_* \LL_1)^2$ as the relations.
Its connection with the Koszul complexes of $\Lambda$ is made via the formulae
$$ D_*(\Lambda^!_t) = K^t .$$

Let $\mu: \Lambda^!_{t-1} \otimes\Lambda^!_1 \longrightarrow \Lambda^!_t$ be the multiplication map, $\partial$ be its dual map, then
$$\partial: D_* \Lambda^!_t \longrightarrow D_* (\Lambda^!_{t-1} \otimes\Lambda^!_1) =D_* \Lambda^!_1 \otimes D_* \Lambda^!_{t-1} = \LL_1 \otimes K^{t-1},$$
by identifying $D_* \Lambda^!_1$ with $\Lambda_1$.

\bigskip

Thus the left Koszul complex of $\Lambda$ is isomorphic to
$$\cdots\longrightarrow \Lambda \otimes D_*(\Lambda^!_2) \longrightarrow \Lambda \otimes D_*(\Lambda^!_1)\longrightarrow \Lambda \otimes D_*(\Lambda^!_0) =\Lambda \longrightarrow 0, $$
where the differential $\partial: \Lambda \otimes D_*(\Lambda^!_t) \longrightarrow \Lambda \otimes D_*(\Lambda^!_{t-1})$ is induced by the dual of the multiplication map $\Lambda^!_{t-1} \otimes\Lambda^!_1 \longrightarrow \Lambda^!_t$, and the identification of $D_*(\Lambda^!_1)$ with $\Lambda_1$.
Note that $\Lambda \otimes D_*(\Lambda^!)$ is a right $\Lambda \otimes {\Lambda^!}^{op}$-module with action defined by $(x\otimes y)(a \otimes a') = x a\otimes a' y $.
Write $u= \sum_{\alpha \in Q_1} \alpha \otimes \alpha^*$ formally, write $(e\otimes e) u = \sum_{\alpha \in Q_1, e \alpha \neq 0} \alpha \otimes \alpha^* \in  \Lambda \otimes \Lambda^!$ for any idempotent $e$ in $\Lambda_0 = D_* \Lambda^!_0 $.
$(e\otimes e) u\in \Lambda \otimes D_*(\Lambda^!)$.
Define $(\sum_i x_i\otimes y_i) u =(\sum_i x_i\otimes y_i) ((e\otimes e) u )$ for $x_i \in e \Lambda e  $.
Then $(\Lambda \otimes D_*(\Lambda^!_t)) u \subset \Lambda \otimes D_*(\Lambda^!_{t-1})$ and right multiplication by $u$ defines exactly the differential of the Koszul complex.

If $\Lambda$ is $(p,q)$-Koszul, we have that $\Lambda^!_t = 0$ for $t > q$.
The only non-zero homology of the above  complex is $\LL_0 e_i$ in degree $0$ and $ \Lambda_p \otimes D_*( \Lambda^!_q )e_i $ in degree $p + q$.

The above complex decomposes into $p+q+1$ homogeneous sub-complexes as (see \cite{bbk})
\tiny
$$
\arr{ccccccccc}{ \Lambda_0 \otimes D_*(\Lambda^!_q)e_i& & \cdots & &  \Lambda_0  \otimes D_*(\Lambda^!_2) e_i & & \Lambda_0  \otimes  D_*(\Lambda^!_1) e_i& &  \Lambda_0 \otimes D_*(\Lambda^!_0)e_i =\LL_0 e_i\\
&\searrow&&\searrow&&\searrow&&\searrow&\\
\Lambda_1 \otimes D_*(\Lambda^!_q)e_i& & \cdots & &  \Lambda_1  \otimes D_*(\Lambda^!_2) e_i & & \Lambda_1  \otimes  D_*(\Lambda^!_1) e_i& &  \Lambda_1 \otimes D_*(\Lambda^!_0)e_i \\
&\searrow&&\searrow&&\searrow&&\searrow&\\
\Lambda_2 \otimes D_*(\Lambda^!_q)e_i& & \cdots & &  \Lambda_2  \otimes D_*(\Lambda^!_2) e_i & & \Lambda_2  \otimes  D_*(\Lambda^!_1) e_i& &  \Lambda_2 \otimes D_*(\Lambda^!_0)e_i \\
&\searrow&&\searrow&&\searrow&&\searrow&\\
\vdots &\vdots&&&\vdots&&\vdots&&\vdots\\
&\searrow&&\searrow&&\searrow&&\searrow&\\
\Lambda_p \otimes D_*(\Lambda^!_q)e_i& & \cdots & &  \Lambda_p  \otimes D_*(\Lambda^!_2) e_i & & \Lambda_p  \otimes  D_*(\Lambda^!_1) e_i& &  \Lambda_p \otimes D_*(\Lambda^!_0)e_i.
}
$$
\normalsize

The total space of this complex can be viewed as a bigraded $\Lambda$-$\Lambda^!$-bimodule with a differential of bidegree $(1,-1)$ which has non-zero homology only in bidegrees $(0, 0)$ and $(p, q)$.
Apply the exact functor
$\bigoplus_{i\in Q_0} \hhom_{\LL_0}(\hhom_{\LL_0}(e_i \LL_0, \quad), {}_{\LL_0}\LL_0)$, we obtain a bigraded $\Lambda^!$-$\Lambda$-bimodule, with a differential of bidegree $(1,-1)$ which has non-zero homology only in bidegrees $(0, 0)$ and $(q, p)$, and which is the total space of the complex
$$ 0\longrightarrow \Lambda^! \otimes D_*\Lambda_p e_i \longrightarrow \cdots \longrightarrow \Lambda^! \otimes D_*\Lambda_{1} e_i\longrightarrow \Lambda^! \otimes D_*\Lambda_{0} e_i \longrightarrow  0.
$$
This complex has non-zero homology only in degree $0$, where it is $\Lambda^!_0 e_i$, and possibly in degree $p + q$,  which is $\Lambda^!_q \otimes D_* \Lambda_p e_i$. Together with the fact that $\Lambda^!_t = 0$ for all $t > q$, this proves the following result from \cite{bbk}.

\begin{pro}\label{pqalgdual}
If $\Lambda$ is a $(p, q)$-Koszul algebra with $p, q \ge 2$, then $\Lambda^!$ is a $(q, p)$-Koszul algebra.
\end{pro}

The facts that the left Koszul complex  provides the first $q+1$ terms of the minimal projective resolution of $\LL_0$ as a left $\Lambda$-module and the above isomorphism of complexes mean that $$\Lambda^!_t \simeq \bigoplus_{i,j \in Q_0} \Ext_{\Lambda}^t (\LL_0 e_i, \LL_0e_j),$$ for $0\le t \le q$.
Write $$\mathcal E(\Lambda) = \bigoplus_{t=0}^q \bigoplus_{i,j \in Q_0} \Ext_{\Lambda}^t (\LL_0 e_i, \LL_0 e_j),$$
Then we have that $\mathcal E(\Lambda) = {\Lambda^!}^{op}$, we call $\mathcal E(\Lambda)$ {\em quadratic dual} of $\Lambda$.

Recall that the Yoneda-Ext algebra $E(\Lambda)$ is defined as the vector space $$ E(\Lambda) = \bigoplus_{t=0}^{\infty} \bigoplus_{i,j\in Q_0} \Ex{t}{S(i), S(j)},$$
with the multiplication defined by the Yoneda product.
It follows from \cite{bbk} that the following hold:
\begin{pro}\label{Yoneda}
If $\Lambda$ is a left $(p, q)$-Koszul algebra with $p > 1$, then $$\mathcal E(\mathcal E(\Lambda))=\Lambda$$ and $\mathcal E(\Lambda)$ is a subalgebra of $E(\Lambda)$ generated by $\mathcal E_0(\Lambda)$ and $\mathcal E_1(\Lambda)$.

Furthermore, $E(\Lambda)$ is generated by $\mathcal E(\Lambda)$ and $ \bigoplus_{i,j \in Q_0} \Ext_{\Lambda}^{q+1} (\LL_0 e_i, \LL_0 e_j)$.
\end{pro}

Let $\Lambda$ be a $(p,q)$-Koszul algebra.
Define $$\mathcal E(M) = \bigoplus_{t=0}^q \bigoplus_{ i, j \in Q_0} \Ext_{\Lambda}^t (M, \LL_0 e_j), $$ for $M \in \grm \Lambda$.

Thus, for each $i$, $$\mathcal E(\Lambda)_t e_i = \bigoplus_{j \in Q_0} \Ext_{\Lambda}^t (\LL_0 e_i, \LL_0 e_j),$$
for $0\le t \le q+1$.

\section{$n$-translation quivers and $n$-translation algebras}

Let $n$ be a non-negative integer, and let $Q = (Q_0, Q_1, \rho)$ be a bound quiver with {\em homogeneous relations}, that is, elements in $\rho$ are linear combinations of the paths of the same length.
Assume that a bijective map $\tau: Q_0 \setminus \mathcal P \longrightarrow Q_0 \setminus \mathcal I$ is defined for two subsets $\mathcal P, \mathcal I \subset Q_0$.
$Q$ is called an {\em $n$-translation quiver} if it satisfies the following conditions:

\begin{enumerate}
\item[1.] For each $i$ in $Q_0 \setminus \mathcal P $, there is a bound path of length $n+1$ form $\tau i$ to $i$, and any maximal bound path is of length $n+1$ from  $\tau i$ in $Q_0 \setminus \mathcal I$ to $i$, for some vertex $i$ in $Q_0 \setminus \mathcal P $.

\item[2.] Two bound paths of length $n+1$ from $\tau i$ to $i$ are linearly dependent, for any $i \in Q_0 \setminus \mathcal P$.

\item[3.] For each $i \in Q_0 \setminus \mathcal P$ and $j\in Q_0$, any bound element $u$ which is linear combination of paths of the same length $t\le n+1$ from $j$ to $i$, there is a path $q$ of length $n+1-t$ from $\tau i$ to $j$ such that  $uq$ a bound element.

\item[4.] For each $i \in Q_0 \setminus \mathcal I$ and $j\in Q_0$, any bound element $u$ which is linear combination of paths of the same length $t\le n+1$ from $i$ to $j$, there is a path $p$ of length $n+1-t$ from $j$ to $\tau i$ such that  $uq$ a bound element.
\end{enumerate}

In this case $\tau$ is called the {\em $n$-translation} of $Q$, we also write it as $\tau_{[n]}$ when we need to specify $n$.
The vertices in $\mathcal P$ are called {\em projective vertices} and vertices in $\mathcal I$ are called {\em injective vertices}.

\begin{lem}\label{nondegbil}
Assume that $Q$ is a bound quiver satisfying the first two condition in the definition.
Let $\LL = k(Q)$ be the algebra given by the bound quiver.
Then 3. and 4. is equivalent to the following condition:
\begin{enumerate}
\item[{\em (5).}] For each $i\in Q_0\setminus \mathcal P$ and each $j\in Q_0$ such that there is a bound path from $j$ to $i$ of length $t$ or there is a bound path from $\tau i$ to $j$, then the multiplication of $\LL$ defines a non-degenerated bilinear from $e_i\LL_t e_{j} \times e_j\LL_{n+1-t} e_{\tau i} $ to $e_i\LL_{n+1} e_{\tau i} $.
\end{enumerate}
\end{lem}

It is clear that if a bound $Q$ is an $n$-translation quiver with  $n$-translation $\tau: Q_0\setminus \mathcal P \to  Q_0\setminus \mathcal I$, then the opposite quiver $Q^{op}$ is also an $n$-translation quiver with  $n$-translation $\tau^{-1}: Q_0\setminus \mathcal I \to  Q_0\setminus \mathcal P$.

\begin{exa}\label{exaa}
{\em
\begin{enumerate}
\item A $0$-translation quiver is a bound quiver whose connected components are either a subquiver of the quiver of type $A_{ \infty }^{ \infty}$, or the quiver of type $\widetilde{ A_l }$, with linear orientation and  all the paths of length $2$ as relations.

\item A bound quiver with each bound path of length $\le n$ is an $n$-translation quiver with $\mathcal P =\mathcal I = Q_0$, we call such $n$-translation quiver {\em null translation quiver}.

    An $m$-translation quiver for $m < n$ is a null $n$-translation quiver.

\item $1$-translation quiver is the usual translation quiver.
\end{enumerate}
}
\end{exa}

\begin{pro}
Let $Q$ be a finite $n$-translation quiver, then $\mathcal P =\emptyset$ if and only if $\mathcal I=\emptyset$.
\end{pro}

In this case, the $n$-translation quiver $Q$ is just {\em a stable bound quiver} of Loewy length $n+2$ defined in \cite{g12}.
We call  an $n$-translation quiver a {\em stable $n$-translation quiver} if the $n$-translation and its inverse are defined everywhere, or equivalently $\mathcal P =\emptyset =\mathcal I$.

A graded algebra $\Lambda$ is called an {\em $n$-pretranslation algebra} if its bound quiver is an $n$-translation quiver.
$0$-pretranslation algebras are uniserial algebras with radical squared zero.
$1$-pretranslation algebras are the translation algebras defined in \cite{g02}.

A stable bound quiver of Loewy length $n+2$ is an $n$-translation quiver and a graded self-injective algebra of Loewy length $n+2$ is an $n$-pretranslation algebra.
We have the following proposition (see \cite{g12}).

\begin{pro}\label{self-alg} An $n$-pretranslation algebra $\Lambda$ is self-injective if and only if its quiver is stable.
\end{pro}

An $n$-pretranslation algebra $\Lambda$ is called an {\em $n$-translation algebra} if there is an $l \in \mathbb N \cup \{\infty \}$ such that $\Lambda$ is $(n+1,l)$-Koszul.
We call $l$ the {\em generalized Coxeter number} of $\Lambda$, as is indicated in the next proposition.
%The bound quiver of an $n$-translation algebra is called a {\em proper $n$-translation quiver.}

Since  $0$-pretranslation algebras are Koszul, they are $0$-translation algebras.

$1$-pretranslation algebras are algebras with the vanishing radical cube. % and so they are also $1$-translation algebra, .
The  self-injective $1$-translation algebras are just self-injective
algebras with the vanishing radical cube.
Their relationship with Koszulity is described in \cite{m1}, see also corollary 4.3 of \cite{bbk}, we rewrite these properties in the language of $1$-translation algebra as follows.

\begin{pro}\label{1-selfc}
Let $\Lambda$ be a connected stable $1$-translation algebra, and let $l$ be its generalized Coxeter number. Then either

\begin{enumerate}
\item $\Lambda$ is representation infinite and $l=\infty$, or

\item $\Lambda$ is representation finite and $l=h-1$, where $h$ is the Coxeter number of its Dynkin diagram.
\end{enumerate}
\end{pro}

\section{Trivial extensions of $n$-translation algebras}
Recall that the trivial extension $\Lambda\ltimes M$ of an algebra $\Lambda$ by a $\Lambda$-bimodule $M$ is the algebra defined on the vector space $\Lambda\oplus M$ with the multiplication defined by $$(a,x)(b,y)=(ab, ay+xb)$$ for $a,b\in \Lambda$ and $x,y\in M$.
The trivial extension $\tilde{\Lambda} = \Lambda\ltimes D_*\Lambda$ of the $\Lambda$ by the bimodule $ D_*\Lambda$ is called {\em the  trivial extension} of $\Lambda$.

Let $Q$ be an $n$-translation quiver with the translation $\tau$.
For each $i \in Q_0\setminus \mathcal P$, fix a path $p_i$ of length $n+1$ from $\tau i $ to $i$.
Let $\Lambda$ be the $n$-pretranslation algebra defined by $Q$.
Then we have the following result.

\begin{lem}\label{exttra}
Write $e_j kQ_1 e_i$ for the space spanned by the arrows from $i$ to $j$.
Then for any $i,j \in Q_0\setminus  \mathcal P$, there is an isomorphism of vector spaces $\tau(i,j): e_j kQ_1 e_i \to e_{\tau j} kQ_1 e_{\tau i} $ such that for any path $p$ of length $n$, and arrow $\alpha \in e_j kQ_1 e_i $, we have in $\Lambda$ $$ \alpha p =p_j \mbox{ and } p \tau(i,j)(\alpha)=p_i.$$
\end{lem}
\begin{proof} Identify $kQ_0+kQ_1$ with its image in $\Lambda$.
Let $\alpha_1, \cdots, \alpha_s$ be the arrows from $i$ to $j$, then by Lemma~\ref{nondegbil}, there are linear combinations $q(1), \cdots, q(s)$ of the paths of length $n$ from $\tau j$ to $i$ such that for $1\le t \le s$, $\alpha_t q(t) = p_j$ and $\alpha_t q(t') =0$ for $t\neq t'$ in $\Lambda$.
Similarly, since $q(1), \cdots, q(s)$ are linear combinations of the paths of length $n$ from $\tau j$ to $i$ and they are linearly independent, there are uniquely determined linear combinations $z_1,\ldots, z_s$ of arrows from $\tau i$ to $\tau j$, such that  for $1\le t \le s$, $q(t)z_t = p_t$ and $q(t') z_t =0$ for $t\neq t'$ in $\Lambda$.
Define $\tau(i,j) \alpha_t = z_t$ for $t=1, \ldots,s$, it is easy to see that $\tau(i,j)$ is the desired map.
\end{proof}
Clearly, $\tau(i,j)$ is in fact an isomorphism of the vector spaces, for any pair $i,j\in Q_0\setminus  \mathcal P$.
Set $\tau(i,j)=0$ for $(i,j)\not\in (Q_0\setminus  \mathcal P)\times (Q_0\setminus \mathcal P)$, and let $$\tau=\bigoplus_{i,j\in Q_0}  \tau(i,j) : \bigoplus_{i,j\in Q_0}  e_j kQ_1 e_i \to \bigoplus_{i,j\in Q_0}  e_j kQ_1 e_i. $$
This extends $\tau$ to $kQ_1$.

Let $p$ be a bound path from $i$ to $j$ and suppose that any bound path from $i$ to $j$ of the same length as $p$ is linearly dependent on $p$, $p$ is called {\em left stark of degree $t$} with respect to $i'$ if $pw$ is a bound element for any bound element from $i'$ to $i$ of length $t< n+1-l(p)$; $p$ is called {\em right stark of degree $t$} with respect to $j'$ if $wp$ is a bound element for any bound element from
 %Changed from "$i'$ to $i$"
 $j$ to $j'$
 of length $t< n+1-l(p)$.
A bound path is call {\em right shiftable} if it is linearly dependent on a set of bound paths passing through no projective vertex, is call {\em left shiftable} if it is linearly dependent on a set of bound paths passing through no injective vertex.
%Write $\bar{p}$ for the left (or right) shiftable version of a bound path $p$.
A bound path $p$ is called {\em semi-shiftable} if it is linear dependent to paths of the form $p'p''$ with $p''$ passing through no injective vertex and $p'$ passing through no projective vertex.
A bound path is call {\em shiftable} if it is left shiftable or right shiftable or semi-shiftable.

An $n$-translation quiver $Q$ is called {\em  admissible} if it satisfies the following conditions:
\begin{enumerate}
\item[(i)] For each bound path $p$, there are paths $q'$ and $q''$ such that $q'pq''$ is a bound path of length $n+1$.
\item[(ii)] Any bound path $p$ from a non-injective vertex $i$ to a non-projective vertex $j$ is
    %Changed from "shiftable."
    linearly dependent to shiftable paths.
\item[(iii)] Let $i$ be a non-projective vertex.
     Let $p$ be a bound path
      %added
      ending at $i$
      and let $q$ be a bound path starting at $\tau i$ with $l(p)+l(q) \le n$.
     If $p$ passes a projective vertex and $q$ passes an injective vertex, then $p$ is either left stark with respect to $t(q)$, or $q$ is right stark with respect to $s(p)$, of length $n+1- (l(p)+l(q))$.
%\item[(iv)] If $p,q$ are two bound paths such that $s(p)=s(q)$ and $t(p),t(q)\in \mathcal P$ then there are paths $p',q'$ with $s(p')=t(p), s(q')=t(q)$ such that $p'p, q'q$ are bound paths with $t(p')=t(q')  \in \mathcal P$.
\end{enumerate}

Clearly, a stable $n$-translation quiver is admissible.

The following proposition is a generalization of Theorem 2.4 of \cite{g14}.

\begin{pro}\label{exttraqr}
Let $\Lambda$ be an $n$-pretranslation algebra given by the  admissible $n$-translation quiver $Q$ with translation $\tau$.
Let $\tilde{\Lambda}$ be its trivial extension and $\tilde{Q} = (\tilde{Q}_0, \tilde{Q}_1, \tilde{\rho})$ be its bound quiver.
Then
\begin{enumerate}
\item $\tilde{Q}_0 =Q_0$ and $\tilde{Q}_1 = Q_1 \cup \{\beta_i : i \to \tau i|i \in Q_0\setminus \mathcal P\}$

\item $\tilde{\rho}= \rho \cup \{\beta_{\tau i}\beta_i | i, \tau i \in Q_0\setminus \mathcal P\} \cup \{\tau(\alpha)\beta_i-\beta_j\alpha| \alpha:i\to j \in Q_1, i,j \in Q_0\setminus \mathcal P\}$

\item $\tilde{\Lambda}$ is a stable $(n+1)$-pretranslation algebra with trivial translation.
\end{enumerate}
\end{pro}
\begin{proof}
Note that $(D_*\Lambda)^2 = 0$ in $\tilde{\Lambda}$, thus it is a nilpotent ideal of $\Lambda$ and is contained in the radical of $\tilde{\Lambda}$.
Take a basis $B$ of $\Lambda$ consisting of paths and containing the chosen path $p_i$ of length $n+1$ from $\tau i $ to $i$, for all $i \in Q_0 \setminus \mathcal P$.
Let $B^*$ be the dual basis of $D_* \Lambda$, and let $p^*_i$ be the basic element dual to $p_i$.

For any $i, j \in Q_0$ and $0 \le t \le n+1$, let $q_1, \ldots, q_r$ be the bound paths of length $t$ from $j$ to $i$ in $B$, then there are linear combinations $q'_1, \ldots, q'_r$ of the paths from $\tau i'$ to $j$  and $q''_1, \ldots, q''_r$ of the paths from $i$ to $i'$ such that $q''_s q_s q'_s= p_{i'}$ for $1\le s \le r$.
If $i \not \in P$, we may take $i'=i$ and $q'_1, \ldots, q'_r$ as dual basis of $q_1, \ldots, q_r$ and we have that $q_s^* =q'_s p_i^* $ for $1\le s \le r$.
If $i\in P$ and $j \not \in I$, we may take $i'=\tau^{-1}j $ and $q''_1, \ldots, q''_r$ as dual basis of $q_1, \ldots, q_r$ and we have that $q_s^* = p_i^* q''_s$ for $1\le s \le r$.
Otherwise $i\in P$ and $j \in I$, take $q_1''$ minimal such that $t(q'')=\tau^{-1}i \not \in P$.
By (iii) of the admissibility either we take $q'_1= \ldots =q'_r$ a left stark elements and  $q''_1, \ldots, q''_r$ dual basis of $q_1q'_1, \ldots, q_rq'_r$, or  we take $q''_1= \ldots =q''_r$ a right stark elements and  $q'_1, \ldots, q'_r$ dual basis of $q''_1q_1, \ldots, q''_rq_r$, then we have $q_s^* = q'_s p_i^* q''_s$ for $1\le s \le r$.
For each $p\in B$, fix $p'$ or $p''$ or $p'$ and $p''$ as defined above.
Then $D_*\LL$ is spanned by
$$\arr{cll}{\bar{C}&=&\{p'p^*_{t(p)}|p\in B, t(p) \in Q_0 \setminus \mathcal P\} \\ && \cup \{ p^*_{\tau^{-1} s(p)} p''|p\in B, t(p) = s(p'') \in \mathcal P, s(p) \in Q_0 \setminus \mathcal I\}\\ && \cup \{p' p^*_{t(p'')} p''|p\in B, p''\mbox{ minimal, not left shiftable}, t(p) = s(p'') \in \mathcal P,\\ && \qquad p'\mbox{ not right shiftable},s(p) = t(p')  \in \mathcal I \}}$$
as vector space.
In fact $\bar{C}$ is a basis of $D_*\LL$, by our construction.
Especially, $D_* \Lambda$ is generated by $\{p_i^*|i\in Q_0 \setminus \mathcal P \}$ as a $\Lambda$-$\Lambda$-bimodule.

So $\tilde{\Lambda}$ is generated by $\Lambda_1$ and $\{p_i^*|i\in Q_0 \setminus \mathcal P \}$ over $\Lambda_0 =\tilde{\Lambda}_0$.
Clearly, we have that $e_{j'}p_i^*e_{i'} \neq 0$ if and only if $ i' = i $ and $j' = \tau i$, and $p^*_i$ is linearly independent of the arrows in $Q$, so it can be regarded as an arrows from $i $ to $\tau i$.
Write $\beta_i$ for the arrow represented by $p^*_i$ and regard it as an element of degree $1$.
This shows that the quiver of  $\tilde{\Lambda}$ has the same set of vertices as $Q$ and has the arrow set $\tilde{Q}_1= Q_1\cup\{\beta_i: i \to \tau i| i\in Q_0\}$.

We have $(D_* \Lambda)^2 = 0$, so $p^*_{\tau i} p^*_i  =0$, and by Lemma \ref{exttra}, we have that $\tau(\alpha)p^*_i-p^*_j\alpha$ for any arrow $\alpha:i\to j$ with $i, j\in Q_0\setminus \mathcal P$.
Thus we have an epimorphism $\psi$ from $k\tilde{Q}/(\tilde{\rho})$ to $\tilde{\Lambda}$ sending $kQ/(\rho)$ to $\Lambda$ and $\beta_i$ to $p_i^*$.

Now let  $$\arr{cll}{C&=&\{p'\beta_{t(p)}|p\in B, t(p) \in Q_0 \setminus \mathcal P\} \\ && \cup \{ \beta_{\tau^{-1} s(p)} p''|p\in B, t(p) = s(p'') \in \mathcal P, s(p) =\tau t(p'')\in Q_0 \setminus \mathcal I\}\\ && \cup \{p' \beta_{t(p'')} p''|p\in B, p''\mbox{ minimal, not left shiftable}, t(p) = s(p'') \in \mathcal P,\\ && \qquad p'\mbox{ not right shiftable},s(p) = t(p') \in \mathcal I
\}.}$$
Clearly, $|C|=|B|$.

Consider an element of the form $\beta_j q \beta_i $ in $k \tQ /(\tilde{\rho})$, then $q$ is a bound path from  $\tau i$ to $j$ with $\tau i$ is not injective and $i,j$ not projective.
%If $i, j\not\in \mathcal P$ and there is a bound path $q\in \Lambda$ from $\tau i$ to $ j$, we show that $\beta_j q \beta_i =0$.
By (ii) of the admissibility, $q$ is
%added
linearly dependent on shiftable paths.
We may assume that $p$ is
%added
shiftable, that is,  left shiftable, right shiftable or semi-shiftable.
In the third case, $\beta_j q \beta_i $ is linearly dependent to the paths of the form $\beta_j q'q'' \beta_i$, such that all the vertices of $q'$ are non-projective and all the vertices of $q''$ are non-injective.
Thus  $\beta_j q'q'' \beta_i$ is written as a linear combination of paths of the form $w'\beta_{j'}  \beta_{i'}w'' $ in $k\tilde{Q} / (\tilde{\rho})$,  by applying the relations of the form $\tau (\alpha) \beta_{\tau^{-1}s(\alpha)} - \beta_{t(\alpha)} \alpha $.
Similarly, $\beta_j q \beta_i $  is written as a linear combination of paths of the form $w'\beta_{j'}  \beta_{i'}$ in $k\tilde{Q} / (\tilde{\rho})$, in the first case and is written as a linear combination of paths of the form $\beta_{j'}  \beta_{i'}w''$ in the second.
This shows that $\beta_j q \beta_i = 0$ for any path $q$ in $k\tilde{Q} / (\tilde{\rho})$, since $\beta_{j'}  \beta_{i'} \in (\tilde{\rho})$ for any $i',j'$.

Thus each element in $k\tilde{Q}/(\tilde{\rho})$ can be written as a linear combination of the form $q'\beta_i q$ with $q'$ a bound path starting from $\tau i$ and $q$ a bound path ending at $i$.
If there $q$ is right shiftable, then $q'\beta_i q$ is written as a linear combination of the  paths of the form $p'\beta_{j''}$, due to the relations of the form $ \tau(\alpha) \beta_{i'} - \beta_{j'} \alpha$  for each arrow $\alpha$ with non-projective $i',j'$.
Similarly if $q' $ is left shiftable, then $q'\beta_i q$ is written as a linear combination of the  paths of the form $\beta_{j''}p''$.
Otherwise $q'$ passes a projective vertex and $q$ pass an injective vertex, by apply the relations of the form $ \tau(\alpha) \beta_{i'} - \beta_{j'} \alpha$ for non-projective $j'$, we may write $q'\beta_i q$ as a linear combination of $p'\beta_{i'}p''$ with $p''$ minimal.

%Otherwise, there is a path $p\in B$ such that $qpq'$ is a bound path of length $n+1$,  passes a projective vertex $i'$ and $p'$ passes an injective vertex, and we may assume that $\alpha:\tau i \to i'$

This shows that $k\tilde{Q}/(\tilde{\rho})$ is spanned by $B$ and $C $ and thus is isomorphic to $\tilde{\Lambda}$.

Set $\{p_i^*|i\in Q_0 \setminus \mathcal P \}$ as elements of degree $1$, then $\tilde{\Lambda}_1 = \Lambda_1 + \sum_{i\in Q_0}k p_i^*$, and we have that $\tilde{\Lambda} =\tilde{\Lambda}_0 + \tilde{\Lambda}_1 + \cdots + \tilde{\Lambda}_{n+2}$, with $\tilde{\Lambda}_t = \tilde{ \Lambda }_1^t $ for $1 \le t \le n+2$.

$\tilde{\Lambda} \simeq k\tilde{Q}/(\tilde{\rho})$ is a symmetric algebra.
It is clearly that each maximal path in $k\tilde{Q}/(\tilde{\rho})$ contains an arrow $\beta_j$ for some $j\in Q_0 \setminus \mathcal P $, and thus has length $n+2$ and has the same vertex as its starting and ending vertex.
Thus $\tilde{\Lambda} \simeq k\tilde{Q}/(\tilde{\rho})$ is a stable $(n+1)$-pretranslation algebra with trivial translation.
\end{proof}

If $\Lambda$ is a Koszul self-injective algebra, then all its projective modules have the same Loewy length, say $n+2$, in this case  $\Lambda$ is a stable $n$-translation algebra with an admissible $n$-translation quiver.
So we have the following theorem in this case.

\begin{thm}\label{triself}
Let $\Lambda$ be a Koszul self-injective algebra. If $\Lambda$ is an $n$-translation algebra, then $\tilde{\Lambda}$ is an $(n+1)$-translation algebra.
\end{thm}
\begin{proof}
Note that the Loewy length of a graded self-injective algebra is one plus the length of its paths of the maximal length.
By Lemma 3.2 of \cite{g14}, the trivial extension of a Koszul self-injective algebra of Loewy length $n+2$ is Koszul self-injective of Loewy length $n+3$.
The Theorem follows directly from the definition.
\end{proof}

If the trivial extension of an $n$-translation algebra $\Lambda$ is an $(n+1)$-translation algebra with an admissible quiver, then $\Lambda$ is called {\em extendable}.

\section{Smash products with cyclic groups and $(n+1)$-translation quiver $\mathbb Z_v|_n Q$}
Let $\mathbb Z_v =\mathbb Z/v \mathbb Z $ be a cyclic group of $v$ elements for a nonnegative integer $v$, and we write $\mathbb Z_0 = \mathbb Z$, conventionally.
Let $\Lambda = \Lambda_0 + \Lambda_1 + \cdots$ be a graded algebra.
We assume that $\Lambda $ has a $\mathbb Z_v$ grading which is determined by a homogeneous subspace $U\subset \Lambda_1$ with a complement $U'$ in $ \Lambda_1$.
Let $W_0 $ be the subalgebra generated by $\Lambda_0$ and $U'$.
Define $W_1=W_0U+UW_0$, and define $W_r$ inductively by
$$W_r = \sum_{s=0}^{r-1}W_{s} U  W_{r-1-s}$$ for $r = 2, 3, \cdots$.
Let $\Lambda'_r = \sum_{t\in r+v \mathbb Z} W_t$.
Then $\Lambda = \Lambda'_0 + \Lambda'_1 +\cdots $ is a $\mathbb Z_v$-graded algebra, we call this a $\mathbb Z_v$-grading determined by the subspace $U$ of $\Lambda_1$.

Recall that the smash product $\Lambda \# k \mathbb Z_v^*$ of $\Lambda$ with $\mathbb Z_v$, defined on the $\mathbb Z_v$-grading of $\Lambda$ determined by $U$, is the free $\Lambda$ module $\Lambda \otimes_k k \mathbb Z_v^*$ with basis  $\mathbb Z_v^* = \{ \delta_g | g \in \mathbb Z_v\}$, and the multiplication is defined by
$$(x\otimes  \delta_h)(y\otimes  \delta_g) = xy'_{h-g}\otimes \delta_g $$
for $x,y \in \Lambda$, $y=\sum_{h\in \mathbb Z_v} y'_h$ with $y'_h \in \Lambda'_h$.
Write $a \delta_g= a\otimes \delta_g$ for the element of $\Lambda \# k\mathbb Z_v^*$.
For any $\Lambda$-module $M$ which has a $\mathbb Z_v$-gradation $M = \sum_{h\in \mathbb Z_v} M'_{h}$, the $\Lambda\# k\mathbb Z_v^*$-module  $M\# k\mathbb Z_v^*$ is defined as $M \otimes_k k\mathbb Z_v^*$ with $$ (a\otimes \delta_g)(m\otimes \delta_h) = am'_{g-h}\otimes \delta_h$$ for all $a\in  \Lambda, m =\sum_{h \in \mathbb Z_v} m'_h \in  M'_h$ and $g, h\in \mathbb Z_v$.
Similarly, we write its elements simply as $m\delta_h$ for $m\in M$ and $h\in \mathbb Z_v$

The original gradation of $\Lambda$ induces a grading on $\Lambda\# k \mathbb Z_v $: $$\Lambda\# k \mathbb Z_v = \Lambda_0 \otimes_k k \mathbb Z_v + \Lambda_1 \otimes_k k \mathbb Z_v \cdots ,$$
with $\{e_i \delta_h | i\in Q_0, h \in \mathbb Z_v\}$ a complete set of orthogonal primitive idempotents, which forms a $k$-basis of $\Lambda_0 \otimes_k k \mathbb Z_v$.
Let $\Lambda_1= \Lambda'_1 + \Lambda''_1$ as a direct sum of vector spaces.
Note that $\Lambda$ is generated by $\Lambda_0$ and $\Lambda_1$, so  $\Lambda_r$ is spanned by the elements of the form $x_1 \cdots x_r$, with $x_t\in \Lambda'_1\cup \Lambda''_1$.
Thus $\Lambda_t\otimes_k  k \mathbb Z_v^*$ is spanned by the elements of the form $x_1 \cdots x_r \delta_h$ with $x_t\in \Lambda'_1\cup \Lambda''_1$ and $h\in \mathbb Z_v$.
But we have that
$$x_1\cdots x_t \delta_h = x_1 \delta_{h_1} \cdots x_t \delta_{h_r}$$
with $h_r =h$ and $h_{t-1} = h_t$ if $x_t\in \Lambda''_1$ and $h_{t-1} = h_t+1$ if $x_t\in \Lambda'_1$.
This shows that $\Lambda\# k\mathbb Z_v^*$ is generated by $\Lambda_0 \otimes_k k \mathbb Z_v$ and $\Lambda_1\otimes k\mathbb Z_v^*$.

Since $\{\delta_{h} | h\in \mathbb Z_v^*\}$ forms a basis of $\Lambda\# k\mathbb Z_v^*$.
If a relation subspace for $\Lambda$ is $R$, then  $R\otimes_k k\mathbb Z_v^*$ is a relation subspace for $\Lambda\# k\mathbb Z_v^*$.

Assume that $\Lambda$ is a quadratic algebra, then $R$ can be chosen as a subspace of $\Lambda_1\otimes_{\Lambda_0} \Lambda_1$.
Thus $R\otimes_k k\mathbb Z_v^* \subset \Lambda_1\# k \mathbb Z_v^*\otimes_{\Lambda_0\# k \mathbb Z_v^*} \Lambda_1\# k \mathbb Z_v^* $ and $\Lambda\# k \mathbb Z_v^*$ is a quadratic algebra.
Recall that the $\Lambda_0 $-$ \Lambda_0 $-sub-bimodules $K^t$ of $\LL_1^{\otimes t}$ is defined by $K^0 = \Lambda_0$, $K^1 = \LL_1$ , $K^2 = R$ and $K^{t+1} = \LL_1 K^t\cap K^t\LL_1$ for all $t \ge 2$.
Define $\Lambda_0\otimes k\mathbb Z_v^* $-$ \Lambda_0 \otimes k\mathbb Z_v^* $-sub-bimodules $\bar{K}^t$ of $(\Lambda_1 \otimes k \mathbb Z_v^* )^{ \otimes_{ \Lambda_0 \otimes k\mathbb Z_v^*} t}$ by the formulae $\bar{K}^0 = \Lambda_0 \otimes_k k \mathbb Z_v^*$, $\bar{K}^1 = \Lambda_1 \otimes_k k \mathbb Z_v^*$ , $\bar{K}^2 = R\otimes_k k \mathbb Z_v^*$ and $\bar{K}^{t+1} = (\Lambda_1 \otimes_k k \mathbb Z_v^*) \bar{K}^t \cap \bar{K}^t (\Lambda_1\otimes_k k \mathbb Z_v^*)$ for all $t \ge 2$.
Similarly, we have the following

\begin{lem}\label{KCl}
$$\bar{K}^{t} = K^{t} \otimes k \mathbb Z_v^*, $$
and
$$ (\Lambda \# k\mathbb Z_v^*)\bar{K}^{t} \simeq \Lambda \otimes K^t \otimes_k k \mathbb Z_v^*.$$
\end{lem}
$K^t$ is a $\Lambda_0$-$\Lambda_0$-bimodule concentrated in degree $t$, thus $\bar{K}^t$ is a $\Lambda_0 \otimes_k k  \mathbb Z_v^* $-$\Lambda_0 \otimes_k k  \mathbb Z_v^* $-bimodule concentrated in degree $t$.
Apply $- \otimes_k k  \mathbb Z_v^* $ on the Koszul complex \eqref{KoszulComplex} of $\Lambda$, we get

\begin{pro}\label{SMkoszulC}
\eqqc{KoszulComplexSM}{\cdots \longrightarrow \Lambda \otimes K^t \otimes_k k \mathbb Z_v^* \longrightarrow \cdots \longrightarrow \Lambda \otimes K^1 \otimes_k  k \mathbb Z_v^* \longrightarrow \Lambda \otimes K^0 \otimes_k k  \mathbb Z_v^* \longrightarrow 0 }
is the left Koszul complex of $\Lambda \#k\mathbb Z_v^*$.
\end{pro}

If  $\Lambda$  is  $(p, q)$-Koszul with $p, q \ge 2$, then we easily get the following properties from Proposition \ref{pqkoszul}:
\begin{enumerate}
\item $\Lambda_t\otimes k \mathbb Z_v^*=0$ for all $t > q$,
\item $K^s \otimes k \mathbb Z_v^*=0 $ for all $s > p$, and
\item the only non-zero homology modules of \eqref{KoszulComplexSM} are $\Lambda_0$ in degree $0$, and $\Lambda_p \otimes K^q \otimes k \mathbb Z_v^*$ in degree $p  + q$.
\end{enumerate}

The following theorem follows easily.

\begin{thm}\label{pqkoszulsmeq} Assume that  $\Lambda$ is a quadratic algebra and $p, q \ge 2$.
Then  $\Lambda$ is a $(p, q)$-Koszul algebra if and only if $\Lambda\#  k \mathbb Z_v^*$ is  a $(p, q)$-Koszul algebra.
\end{thm}

If $\Lambda$ is an $n$-pretranslation algebra with an admissible $n$-translation quiver, the {\em $v$-extension of $\Lambda$} is defined as  the smash product $\tilde{\Lambda}\# \mathbb Z_v^*$ of the trivial extension $\tilde{\Lambda}$ of $\Lambda$ with $\mathbb Z_v^*$, defined on the $\mathbb Z_v^*$-grading determined by the vector spaces spanned by the new arrows of the trivial extension.

As a corollary of Theorem \ref{pqkoszulsmeq}, we have the following proposition.

\begin{pro}\label{extendible}
Let $\Lambda$ be an extendable $n$-translation algebra with an admissible $n$-translation quiver, then its $v$-extension is an $(n+1)$-translation algebra.
\end{pro}

Let $Q$ be an admissible $n$-translation quiver with $n$-translation $\tau=\tau_{[n]}$, and projective vertex set $\mathcal P$ and injective vertex set $\mathcal I$.
We define the $n+1$ translation quiver $\mathbb Z_v|_n Q$ as follows.

The vertex set $$\mathbb Z_v|_n Q_0 = Q_0 \times \mathbb Z_v =\{(i , t)| i\in Q_0, t \in \mathbb Z_v\},$$
and the arrow set
$$\arr{cll}{\mathbb Z_v|_n Q_1 & =& \mathbb Z_v \times Q_1 \cup \mathbb Z_v \times (Q_0\setminus \mathcal P)\\ &=& \{(\alpha,t): (i,t)\longrightarrow (j,t) | \alpha:i\longrightarrow j \in Q_1, t \in \mathbb Z_v.\}\\ && \cup \quad\{(\beta_i , t): (i, t-1) \longrightarrow (\tau i, t) | i \in Q_0\setminus \mathcal P, t \in \mathbb Z_v.\}
}$$

Let $t\in \mathbb Z_v$.
Denote by $p(t)$ the path $$p(t)= (\alpha_l,t) \cdots (\alpha_1,t) $$ for each path $p= \alpha_l \cdots \alpha_1 $ in the quiver $Q$ and denote by $z(t)$ the linear combination $$z(t) = \sum_{s} a_sp_s(t) $$ for each linear combination of the paths $z = \sum_{s} a_sp_s $ in the quiver $Q$.
Let $Q(t)$ be the subquiver with vertex set $Q(t)_0= \{(i, t) | i \in  Q_0\}$ and arrow set $Q(t)_1= \{(\alpha, t) | \alpha  \in  Q_1 \}$ and relation set $\rho(t)= \{z(t) | z \in \rho\}$, for each $t\in \mathbb Z$.
Then $Q(t) = (Q(t)_0, Q(t)_1, \rho(t))$ is a subquiver of $\mathbb Z_v|_n Q$ isomorphic to $Q$ as $n$-translation quivers.

Let $ \rho(t)^* = \{(\tau(\alpha), t+1) (\beta_i, t)-(\beta_j, t)(\alpha, t) |i, j\in Q_{0}\setminus \mathcal P, \alpha: i \to j \in Q_1\} \cup \{(\beta_{\tau i}, t+1) (\beta_i, t) | i, \tau i\in Q_0\setminus \mathcal P\},$ take $$\rho_{ \mathbb Z_v|_n Q} = \cup_{t\in \mathbb Z_v} (\rho(t)\cup \rho(t)^*) $$
as the relation set for the bound quiver $\mathbb Z_v|_n Q $.

Define $\tau_{[n+1]}: (i,t) \longrightarrow (i,t-1)$ for $(i,t) \in \mathbb Z_v|_n Q_0 $.

By Proposition \ref{exttraqr}, we know that the trivial extension $\tilde{\Lambda}$ is an $(n+1)$-pretranslation algebra with stable $(n+1)$-translation quiver $\tilde{Q} = (\tilde{Q}_0, \tilde{Q}_1, \tilde{\rho})$ with trivial translation.
This quiver has the same vertex set as $Q$, arrow set $\tilde{Q}_1 = Q_1 \cup \{\beta_i : i \to \tau_{[n]} i |i \in Q_0\setminus \mathcal P\}$ and a relation set $\tilde{\rho}= \rho \cup \{\beta_{\tau_{[n]} i}\beta_i | i, \tau_{[n]} i \in Q_0\setminus \mathcal P\} \cup \{\tau_{[n]}(\alpha)\beta_i-\beta_j\alpha| \alpha:i\to j \in Q_1, i,j \in Q_0\setminus \mathcal P\}.$
Here $\beta_i = p_i^*$ for some fixed bound path $p_i$ of maximal length in $\Lambda$ for each $i \in Q_0\setminus \mathcal P$.
Let $\tilde{\Lambda}_1$ be the space spanned by $\Lambda_1$ and $\{\beta_i : i \to \tau_{[n]} i |i \in Q_0\setminus \mathcal P\}$, we have that $\tilde{\Lambda} =\tilde{\Lambda}_0 + \tilde{\Lambda}_1 + \cdots + \tilde{\Lambda}_{n+1}$, with $\tilde{\Lambda}_t = \tilde{ \Lambda }_1^t $ for $1 \le t \le n+1$.

Let $U$ be the subspace of $\tilde{\Lambda}_1$ spanned by the set $\{\beta_i : i \to \tau_{[n]} i |i \in Q_0\setminus \mathcal P\}$ and let $\tilde{\Lambda}\# \mathbb Z_v$ be the smash product defined on the $\mathbb Z_v$-grading determined by $U$.

\begin{pro}\label{Zinfcon}
Let $Q$ be an admissible $n$-translation quiver, and let $\Lambda $ be the $n$-pretranslation algebra defined by $Q$.
Then $\mathbb Z_v|_n Q $ is the bound quiver for $\tilde{\Lambda}\# \mathbb Z_v^*$.
Especially  $\mathbb Z_v|_n Q$ is a stable $(n+1)$-translation quiver with translation $\tau_{[n+1]}$.
\end{pro}
\begin{proof}
Let $\overline{Q}=(\overline{Q}_0,\overline{Q}_1,\overline{\rho})$ be the bound quiver of $\tilde{\Lambda}\#\mathbb Z_v^*$.
Clearly, we have that $\tilde{\Lambda}\#\mathbb Z_v^* =\tilde{\Lambda}_0  \otimes_k \mathbb Z_v^* + \tilde{\Lambda}_1  \otimes_k \mathbb Z_v^* + \cdots + \tilde{\Lambda}_{n+1}  \otimes_k \mathbb Z_v^*$.
Note that $\{e_i\delta_n|i\in Q_0, n\in\mathbb Z_v^*\}$ is a basis of $\tilde{\Lambda}_0 \otimes_k \mathbb Z_v^*$, and we have that $$e_i\delta_{m'}\cdot e_j\delta_m = \left\{\arr{ll}{e_i\delta_{m} & i=j , m=m'\\0& otherwise} \right. ,$$
and $e_i\delta_n \tilde{\Lambda}\#\mathbb Z_v^* e_i\delta_n\simeq e_i \tilde{\Lambda} e_i \simeq e_i \Lambda e_i \simeq k$.
On the other hand, $(\tilde{\Lambda}_0 \otimes_k \mathbb Z_v^*) \cdot (\tilde{\Lambda}_t  \otimes_k \mathbb Z_v^*) =(\tilde{\Lambda}_1  \otimes_k \mathbb Z_v^*) \cdot (\tilde{\Lambda}_0 \otimes_k \mathbb Z_v^*) = \tilde{\Lambda}_1  \otimes_k \mathbb Z_v^* $, so $\{e_i\delta_n|i\in Q_0, n\in\mathbb Z_v^*\}$ is a complete set of orthogonal primitive idempotents.
Write $(i,n) = e_i\delta_n$, thus we have $\overline{Q}_0 = \mathbb Z_v|_n Q_0$.

Since $ Q_1 \cup \{\beta_i : i \to \tau_{[n]} i |i \in Q_0\setminus \mathcal P\}$ forms a basis of $\tilde{\Lambda}_1$,  $\tilde{\Lambda}_0 \otimes_k \mathbb Z_v^*$ and $( Q_1\times \mathbb Z_v^*)\cup (\{\beta_i : i \to \tau_{[n]} i |i \in Q_0\setminus \mathcal P\}\times \mathbb Z_v^*)$ is a basis of  $\tilde{\Lambda}\#\mathbb Z_v^* $, and one see easily that $e_{j'}\delta_{m'} \alpha \delta_m  =\alpha \delta_m e_{i'}\delta_{m''} = \alpha \delta_m$ for any arrow $\alpha:i \to j$ in $Q_1$ if and only if $i'=i, j'=j$ and $m'=m''=m$,  $e_{j'}\delta_{m'} \beta_i \delta_m e_{i'}\delta_{m''}= \beta_i \delta_m$ if and only if $i'=i, j'=\tau_{[n]} i$, $m'=m+1$  and $m''=m$.
Set $(\alpha, m) = \alpha \delta_m$ be an arrow from $(i,m)$ to $(j,m)$ for an arrow $\alpha:i \to j$ and set $(\beta_i, m) = \beta_i \delta_m$ be the arrow from $(i,m)$ to $(\tau_{[n]} i, m+1)$, for each $m\in\mathbb Z$, then we have that $\overline{Q}_1 = \mathbb Z_v|_n Q_1$.

By Proposition \ref{exttraqr}, the relation set for $\tilde{\Lambda}$ is $\tilde{\rho}= \rho \cup \{\beta_{\tau_{[n]} i}\beta_i | i, \tau i \in Q_0\setminus \mathcal P\} \cup \{\tau_{[n]}(\alpha)\beta_i-\beta_j\alpha| \alpha:i\to j \in Q_1, i,j \in Q_0\setminus \mathcal P\}$.
Note that for a path $p=\alpha_r\cdots\alpha_1$ and $m\in \mathbb Z$, we have a unique path $p\delta_m= \alpha_r\delta_m\cdots\alpha_1\delta_m$, we also have $ \alpha\beta_i\delta_m= \alpha\delta_{m+1} \beta_i\delta_m $, $\beta_j\alpha\delta_m = \beta_j\delta_m\alpha\delta_m$ and $\beta_{\tau_{[n]} i}\beta_i\delta_m = \beta_{\tau_{[n]} i}\delta_{m+1} \beta_i\delta_m$.
Write $(p,m)$ for the path $p\delta_m$ and $(z, m)$ for a relation $z\delta_m = \sum a_t p_t \delta_m$ for $z = \sum a_t p_t\in \rho$, we see that $\overline{\rho} = \rho_{ \mathbb Z_v|_n Q}$.
So the bound quiver of $\tilde{\Lambda}\# \mathbb Z_v$ is exactly $\mathbb Z_v|_n Q $.

For each $i\in Q_0$, there is a maximal bound path in $\tilde{\Lambda} $ of the form $p\beta_i q$ with paths $p,q$ of $Q$ such that $l(p)+l(q)=n$ and $t(p)=s(q)=i$.
So for each $(i,m)\in \mathbb Z_v|_n$, we have a maximal bound path in $\tilde{\Lambda} \# \mathbb Z^*$ has the form $p\beta_i q\delta_m$ from $(i,m)$ to $(i,m+1)$.
Thus $\mathbb Z_v|_n Q$ is a stable $(n+1)$-translation quiver with translation $\tau_{[n+1]}: (i,m+1) \to (i,m)$.
\end{proof}

It is well known that for a path algebra $kQ$ of a quiver $Q$ without oriented cycle, the preprojective and preinjective components of its Auslander Reiten quiver are truncations of the translation quiver $\mathbb ZQ$.

Let $Q$ be $0$-translation quiver as in Example ~\ref{exaa}, then $\LL = k(Q)$ is a $0$-translation algebra with admissible $0$-translation quiver.
Let $\overline{\LL} = \tilde{\LL}\# \mathbb Z^*$ be the $0$-extension of $\LL$, then the bound quiver $\overline{\LL}$ is the $1$-translation quiver $\mathbb Z|_1 Q$, which is exactly the classical one, $\mathbb Z Q$.

If $Q$ is the quiver with bipartite orientation (when it exists), then $\LL = kQ$ is an $(1,1)$-Koszul algebra.
Let $\overline{\LL} = \tilde{\LL}\# \mathbb Z^*$ be the $0$-extension of $\LL$, then $\overline{\LL}$ is a $1$-translation algebra with $1$-translation quiver $\mathbb Z|_1 Q$ exactly the translation quiver $\mathbb Z Q$.
In general, $Q$ is a $1$-translation quiver, but not admissible, so our construction does not include the classical cases.

\medskip

Let $\Lambda$ be an extendable $n$-translation algebra with $n$-translation quiver  $Q$, which is directed in the sense that there is no oriented cycle in the quiver.
Let $\tL$ be its trivial extension, then $\tL$ is a graded self-injective algebra with Loewy length $n+3$.
Let $\overline{Q}$ be the  separated directed quiver of $\tL$, then $\mathbb Z|_n Q$ is isomorphic to a connected component of $\overline{Q}$ and thus $Q$ is isomorphic to a connected component $\tau$-slice in $\overline{Q}$.
Thus $\Lambda$ is a $\tau$-slice algebra of $\tL$.
So by Theorem 6.12 of \cite{g12}, $\tL\#\mathbb Z^*$ is isomorphic to the repetitive algebra of $\LL$ and by  Theorem II 4.9 of \cite{h1}.
Thus we have that the following theorem.

\begin{thm} Let $\Lambda$ be an extendable $n$-translation algebra with directed admissible $n$-translation quiver.
$\tL\#\mathbb Z^*$ is the repetitive algebra of $\Lambda$ and the bounded derived category $\mathcal D^b(\Lambda)$ of the category $\bmod\, \Lambda$ of finitely  generated $\Lambda$-modules is equivalent to the stable category $\underline{\bmod}\,\tL \# k \mathbb Z^*$ of finitely generated $\tL \# k \mathbb Z^*$-modules.
\end{thm}

\section{$n$-translation algebras and partial  Artin-Schelter $n$-regular algebras}
Let $\Gamma= \Gamma_0+\Gamma_1+\cdots$ be a basic locally finite algebra, and let $\{e_i|i\in Q_0\}$ be a complete set of orthogonal primitive idempotents of $\Gamma$.
$\Gamma$ is called {\em a partial  Artin-Schelter $n$-regular algebra} if there are subsets $\mathcal P$, $\mathcal I$ of $Q_0$ and bijective map $\nu:Q_0 \setminus\mathcal I \longrightarrow Q_0 \setminus \mathcal P$ such that the following conditions are satisfied for  positive integers $n,l$, and for each vertex $i \in Q_0 \setminus \mathcal I$:
(i) $\Ext_{\Gamma}^t(\Gamma_0 e_i , \Gamma)=0$ for $0 < t < n+1$;
(ii). $\Ext_{\Gamma}^{n+1}(\Gamma_0 e_i , \Gamma)[l] \simeq e_{\nu i}\Gamma_0 $ as  $\Gamma^{op}$-modules.

$\nu$ is called {\em Nakayama translation} and $l$ is called {\em  Gorenstein parameter} of $\Gamma$.
Conventionally, if $\Lambda$ is a $(p,q)$-Koszul algebra with $p$ finite, we set $q=\infty$ when $\Lambda$ is Koszul.
Let $l(q) = q$ if it is finite and $0$ if $q$ is $\infty$.
Using an argument analog to that of Proposition 5.10 in \cite{sm}, we have

\begin{thm}\label{kdl}
Assume that  $p= n+1 \ge 2$ is  finite and $ q\ge 2$. Let $\Lambda$ be a $(p,q)$-Koszul algebra. Then $\Lambda$ is an $n$-translation algebra with the generalized Coxeter number $q$ if and only if $\Gamma= \Lambda^!$ is a  partial  Artin-Schelter $n$-regular algebra with  Gorenstein parameter $l(q)$.
\end{thm}
\begin{proof}
Since $\Lambda$ is a $(p,q)$-Koszul algebra, $\Gamma= \Lambda^!$ is a $(q, p)$-Koszul algebra, by Proposition \ref{pqalgdual}.

If $\Lambda$ is $n$-translation algebra with null $n$-translation quiver, then $\Lambda_p=0$ for $p=n+1$ and $\Lambda$ is Koszul of projective dimension $\le q$, and this holds if and only if $\Gamma= \Lambda^!$ is a Koszul of projective dimension $< p$, and $Q_0= \mathcal I  =\mathcal P$.
This holds if and only if  $Q_0 \setminus\mathcal I=\emptyset$, and $\Gamma$ is a partial Artin-Schelter $n$-regular algebra.

If the bound quiver of $\Lambda$ is not null $n$-translation quiver, then we have  $\Lambda_p \neq 0$.
Thus $\Gamma = \Lambda^!$ is $(q, p)$-Koszul, and for each $i\in Q_0$, we have a Koszul complex
\eqqc{KoszulComplexii}{0\longrightarrow \Gamma \otimes D_*(\Gamma^!_p) e_i \longrightarrow \cdots \longrightarrow \Gamma \otimes D_*(\Gamma^!_{1}) e_i\longrightarrow \Gamma \otimes D_*(\Gamma^!_{0})  e_i = \Gamma  e_i \longrightarrow  0}
with non-zero homologies only in degree $0$, where it is $\Gamma_0 e_i$, and in degree $p + q$, where it is $\Gamma_q \otimes D_*( \Gamma^!_p) e_i$.
This complex is indecomposable since it is in a minimal projective resolution of $\Gamma_0 e_i$.
Applying $\hhom_{\Gamma}(\quad, \Gamma)$ on this complex one gets
\begin{eqnarray}\label{eq1} 0\longrightarrow e_i \Gamma \stackrel{d}{\longrightarrow} e_i\Gamma^!_{1}\otimes\Gamma  \stackrel{d}{\longrightarrow} \cdots \stackrel{d}{\longrightarrow} e_i\Gamma^!_{p-1}\otimes  \Gamma \stackrel{d}{\longrightarrow} e_i \Gamma^!_{p} \otimes \Gamma \longrightarrow  0.
\end{eqnarray}
$\Gamma$ is partial Artin-Schelter $n$-regular if and only if for each $i \in Q_0\setminus \mathcal I$, the only nonzero homologies in \eqref{eq1} are  at two ends, and the one at the final position is $e_{\nu i}\Gamma_0$.
On the other hand, $\Gamma$ is right $(q,p)$-Koszul, so there is a Koszul complex for $ e_{\nu i}\Gamma_0$:
\small\begin{eqnarray}\label{eq2}
0\longrightarrow e_{\nu i} D_*(\Gamma^!_p) \otimes \Gamma \stackrel{\partial}{\longrightarrow} \cdots \stackrel{\partial}{\longrightarrow} e_{\nu i}D_*(\Gamma^!_{1}) \otimes \Gamma \stackrel{\partial}{\longrightarrow} e_{\nu i}D_*(\Gamma^!_{0}) \otimes \Gamma = e_{\nu i}\Gamma \longrightarrow  0.
\end{eqnarray}\normalsize
Since $\Gamma$ is $(q,p)$-Koszul and both complexes \eqref{eq1} and \eqref{eq2} are the first $p+1$ term in a minimal projective resolution of $ e_{\nu i}\Gamma_0$, so there is an isomorphism of complexes
\tiny
$$
\arr{ccccccccc}{
 e_i\Gamma^!_{0}\otimes \Gamma &\stackrel{d}{\longrightarrow} &  e_i\Gamma^!_{1}\otimes\Gamma  & \stackrel{d}{\longrightarrow}& \cdots & \stackrel{d}{\longrightarrow} &e_i\Gamma^!_{p-1} \otimes \Gamma & \stackrel{d}{\longrightarrow} & e_i \Gamma^!_{p} \otimes \Gamma \\
\downarrow_{\phi} &&\downarrow_{\phi} &&&& \downarrow_{\phi}&&\downarrow_{\phi}\\
 e_{\nu i} D_*(\Gamma^!_p) \otimes \Gamma & \stackrel{\partial}{\longrightarrow} & e_{\nu i} D_*(\Gamma^!_{p-1}) \otimes \Gamma & \stackrel{ \partial}{ \longrightarrow} & \cdots & \stackrel{ \partial}{ \longrightarrow} & e_{\nu i}D_*(\Gamma^!_{1}) \otimes \Gamma & \stackrel{ \partial}{ \longrightarrow} & e_{\nu i}D_*( \Gamma^!_{0}) \otimes \Gamma .
}
$$
\normalsize
Here $d$ and $\partial$ are left actions by $u=\sum_{\alpha} \alpha^* \otimes \alpha$ and $\phi = \Phi \otimes \mathrm{1}$ for some isomorphism $\Phi: e_i \Gamma_t \longrightarrow e_{\nu i}D_*( \Gamma_{p-t}) = D_*( \Gamma_{p-t} e_{\nu i}) $, $0\le t \le p$ of graded vector spaces.
Thus $\phi d = \partial \phi $, that is, if $x \otimes y \in \Gamma^! \otimes \Gamma$ for bound paths $x$ and $ y$, then $$\arr{lcl}{ \sum_{\alpha\in Q_1 }(\Phi( x \alpha^*) \otimes \alpha y)  &=& \Phi \otimes \mathrm{1} (\sum_{\alpha\in Q_1 } x \alpha^* \otimes \alpha y ) \\ &=& \phi d (x \otimes y )\\ &=&\partial \phi (x \otimes y )\\ &=&\partial (\Phi (x) \otimes y )  \\ &=& \sum_{\alpha\in Q_1 } (\Phi (x)\alpha^* \otimes \alpha y ).}$$
So $ \Phi (x\alpha^*) =\Phi (x) \alpha^* $ for all $\alpha\in Q_1$ and $\Phi: e_i \Gamma^! \to D_*( \Gamma^! e_{\nu i})$ is an isomorphism of right $\Gamma^!$-modules.
This is equivalent to that $ D_*(e_i \Gamma^!) \simeq \Gamma^! e_{\nu i}$ is projective-injective as left $\Gamma^!$ module, and the injective hull of $S(i)$ and the projective cover of $S({\nu i})$ are the same.
Now take the inverse of $\nu$ as the $n$-translation $\tau$, then $\tau$ is defined for any $ i \in Q_0 \setminus \mathcal P$ and $\tau i \in Q_0 \setminus \mathcal I$, and we have that the bound paths from $\tau i$ to $i$ are maximal ones with length $p$ and any two such paths are linear dependent, and there is no bound path of length $p$ from $\tau i$ to $j$ for any $j \neq i$.

This equivalent to that the bound quiver of $\Lambda=\Gamma^!$  is an $n$-translation quiver and $\Lambda$ is a  $n$-translation algebra, since $n=p-1$.
This proves our theorem.
\end{proof}

Let $\Lambda= k(Q)$ be an $n$-translation algebra  with bound quiver $Q =(Q_0,  Q_1, \rho )$.
Then its quadratic dual $\Gamma = \Lambda^!$ is a partial Artin-Schelter $n$-regular algebra.
Combine Theorem \ref{kdl} with Proposition \ref{self-alg}, we get the first part of the following theorem.

\begin{thm}\label{eqntrans} An $n$-translation algebra $\Lambda$ is self-injective if and only if its quiver is stable, if and only if its quadratic dual $\Gamma$ is a partial Artin-Schelter $n$-regular algebra with $\mathcal P =\emptyset =\mathcal I$.

If $\Lambda$ is Koszul, then  $\Gamma$ is an Artin-Schelter regular algebra of dimension $n+1$.

If $\Lambda$ is not Koszul, then $\Gamma$ is also self-injective.
\end{thm}
\begin{proof}
It follows \cite{sm, m2} that $\Gamma$ is an Artin-Schelter regular algebra of dimension $n+1$ if $\Lambda$ is Koszul.

Assume that $\Lambda$ is not Koszul.
Since $\Lambda$ is right $(p.q)$-Koszul self-injective, for each $i\in Q_0$, we have a Koszul complex
$$ 0\longrightarrow e_i D_*(\LL^!_q) \otimes \LL \longrightarrow \cdots \longrightarrow e_i D_*(\LL^!_1) \otimes \LL \longrightarrow e_i D_*(\LL^!_0) \otimes \LL =  e_i \LL \longrightarrow  0,$$
with non-zero homologies only in degree $0$, where it is $e_i \LL_0$, and in degree $p + q$, where it is $e_i D_*(\LL^!_q) \otimes \LL_p$.
This complex is indecomposable since it is in a minimal projective resolution of $e_i \LL_0 $.
Applying $\hhom_{\LL}(\quad, \LL)$ on this complex one gets
\begin{eqnarray}\label{eq11} 0\longrightarrow \LL e_i \stackrel{d}{\longrightarrow} \LL\otimes\LL^!_{1}e_i \stackrel{d}{\longrightarrow} \cdots \stackrel{d}{\longrightarrow}\LL\otimes  \LL^!_{q-1} e_i\stackrel{d}{\longrightarrow} \LL \otimes \LL^!_{q} e_i\longrightarrow  0.
\end{eqnarray}

$\LL$ is injective as $\LL$-module, the only nonzero homologies in \eqref{eq11} are at two ends, and the one at the final position is $\LL_0 e_{\nu i}$.
On the other hand, $\LL$ is left $(p,q)$-Koszul, so there is a Koszul complex for $ \LL_0 e_{\nu i}$:
\small\begin{eqnarray}\label{eq21}
0\longrightarrow \LL \otimes D_*(\LL^!_q)e_{\nu i} \stackrel{\partial}{\longrightarrow} \cdots \stackrel{\partial}{\longrightarrow} \LL \otimes D_*(\LL^!_1)e_{\nu i} \stackrel{\partial}{\longrightarrow} \LL \otimes D_*(\LL^!_0)e_{\nu i}= \LL e_{\nu i} \longrightarrow  0.
\end{eqnarray}\normalsize
Since $\LL$ is $(p,q)$-Koszul and both complexes \eqref{eq11} and \eqref{eq21} are the first $p+1$ term in a minimal projective resolution of $\LL_0 e_{\nu i}$, similar to the proof of  Theorem \ref{kdl}, that for each $i \in Q_0$,  $\nu i$ is defined and $ D_*(\LL^!e_i ) \simeq e_{\nu i}\LL^!$, that is $ D_*(\Gamma e_i ) \simeq e_{\nu i}\Gamma$.
This proves that $\Gamma$ is self-injective.
\end{proof}

\section{Koszul duality and $n$-almost split sequences}
Assume that  $p= n+1 \ge 2, q\ge 2$ throughout this section.
Now we study the $n$-almost split sequences related to $n$-translation algebras.

Let $\Gamma$ be a locally finite algebra, and let $\mathcal Q =\add \GG$ be the category of the finite generated projective $\GG$-modules.
The minimal projective resolution of the simples gives rise to sink sequences  in $\cQ$.
In fact, if \eqqc{sourseseq}{ 0\longrightarrow P^{n+1} (S(i)) \stackrel{f_{n+1}}{\longrightarrow} \cdots \stackrel{f_2}{\longrightarrow} P^1(S(i)) \stackrel{ f_1}{\longrightarrow} P^0 (S(i))\longrightarrow 0}
is the complex truncated from a minimal projective resolution of a simple $\GG$-module  $S$, then we have an exact sequence
\small$$ \hhom_{\cQ} ( \qquad, P^{n+1}(S(i)) ) \stackrel{f_{n+1}}{\longrightarrow} \cdots \stackrel{f_2}{\longrightarrow}\hhom_{\cQ} ( \qquad, P^1(S(i))) \stackrel{ f_1}{\longrightarrow} J_{\cQ} ( \qquad, P^0) \longrightarrow 0,$$\normalsize and thus
\eqref{sourseseq} is a sink sequence in $\cQ$.
It is interesting to know when it is a source sequence.

\begin{lem}\label{nalmost} Let $\Gamma$ be a $(q,n+1)$-Koszul partial Artin-Schelter $n$-regular algebra.
Assume that the complex \eqref{sourseseq} is truncated from a minimal projective resolution of the simple $\Gamma$-module $S(i)$.
Then \eqref{sourseseq} is a source sequence of $\cQ$ if and only if it is exact at position $n+1$.

In this case, \eqref{sourseseq} is an $n$-almost split sequence.
\end{lem}

\begin{proof}
Assume that $\Gamma$ is a  partial Artin-Schelter $n$-regular Koszul algebra and that the complex \eqref{sourseseq} is a truncation of  a minimal projective resolution of the simple $\Gamma$-module $S(i)$.

Since $\GG$ is a $(q,n+1)$-Koszul, we have a long exact sequence
$$0 \longrightarrow M \stackrel{f_{n+1}}{\longrightarrow} \cdots \stackrel{f_2}{\longrightarrow} P^1(S(i)) \stackrel{ f_1}{\longrightarrow} P^0 (S(i))\longrightarrow S(i)\longrightarrow 0, $$
where $M\simeq P^{n+1} (S(i))/\GG_q e_{\nu i}$ if $q$ is finite and $\GG_{p} e_{\nu i}\neq 0$ and $M\simeq P^{n+1} (S(i))$ otherwise, and \eqref{sourseseq} is exact at position $n+1$ in this case.
We have a short exact sequence  $0 \longrightarrow M \stackrel{}{\longrightarrow}P^{n} (S(i)) \longrightarrow \Omega^{n} (S(i)) \longrightarrow 0 $ and thus
$$\arr{ll}{0 \longrightarrow \hhom_{\GG}(\Omega^{n} S(i),\GG) \longrightarrow \hhom_{\GG}(P^{n} (S(i)),\GG)\\ \qquad\qquad \longrightarrow \hhom_{\GG}(M,\GG) \longrightarrow  \Ext_{\GG}(\Omega^{n} S(i),\GG) \longrightarrow 0 }$$
is exact.
So we have $ \Ext_{\GG}(\Omega^{n} S(i),\GG) \simeq \Ext^{n+1}_{\GG} ( S(i),\GG) $ is simple as $\GG^{op}$-module.
If \eqref{sourseseq} is exact at position $n+1$
Then we have an exact sequences
$$\arr{ll}{\hhom_{\cQ}(P^0(S(i)),\quad)\longrightarrow \cdots  \longrightarrow \hhom_{\cQ}(P^{n-1}(S(i)),\quad)\\ \qquad\qquad \longrightarrow \hhom_{\cQ}(P^{n}(S(i)),\quad)\longrightarrow J\hhom_{\cQ}(P^{n+1}(S(i)),\quad) \longrightarrow 0,}$$
and thus \eqref{sourseseq} is a source sequence and hence is an $n$-almost split sequence in $\cQ$.

Otherwise, we have $$\arr{ll}{\longrightarrow \hhom_{\GG}(P^{n-1}( S(i)),\GG) \longrightarrow \hhom_{\GG}(P^{n} (S(i)),\GG)\\ \qquad\qquad \longrightarrow \hhom_{\GG}(P^{n+1}(S(i))/P^{n+1}(S(i))_t, \GG) \longrightarrow  \Ext_{\GG}(\Omega^{n} S(i),\GG) \longrightarrow 0, }$$
and $\hhom_{\GG}(P^{n+1}(S(i))/P^{n+1}(S(i))_t, \GG)\subseteq J\hhom_{\cQ}(P^{n+1}(S(i)),\quad) $.
Thus  \eqref{sourseseq} is not a  source sequence.
\end{proof}

Let $\Gamma$ be a  partial Artin-Schelter $n$-regular algebra with the Nakayama translation $\nu$, if for $i \in Q_0$ with $\nu i \in Q_0$, \eqref{sourseseq} is an $n$-almost split sequence in $\cQ=\add\GG$, then we say that {\em $\GG$ has $n$-almost split sequences}.
As a corollary of Lemma \ref{nalmost}, we have the following theorem.

\begin{thm}\label{nart} Let $\Gamma$ be a  partial Artin-Schelter $n$-regular $(q,n+1)$-Koszul algebra with Nakayama translation $\nu$.
If $q$ is infinite or $q$ is finite with $\GG_q e_{\nu i}=0$  for any $i\in Q_0 \setminus \mathcal I$, then $\mathcal Q =\add \GG$ has $n$-almost split sequences.
That is, for each $i\in Q_0\setminus \mathcal I$, there is an $n$-almost split sequence   $$ \Gamma e_{\nu i} \to  U_{n} \to \cdots \to U_1 \to \Gamma e_i$$ in $\mathcal Q$.
\end{thm}

In this case, $\tau_{[n]}: \Gamma e_i \to \Gamma e_{\nu i}$ is called the {\em $n$-Auslander-Reiten translation of $\mathcal Q$}.

Let $\Lambda$ be an $n$-translation algebra with bound quiver $Q =(Q_0,  Q_1, \rho )$.
Call $\Gamma= \mathcal E(\Lambda)\simeq {\Lambda^!}^{op}$ {\em the dual $n$-translation algebra} of $\Lambda$.
Then we have that $\LL \simeq \mathcal E(\GG)\simeq {\GG^!}^{op}$, and $\Gamma$ is a  partial Artin-Schelter $n$-regular algebra.
It follows from the proof of Theorem \ref{kdl} the map $\nu$ for $\Gamma$ is just the inverse of the $n$-translation $\tau$ for $\Lambda$.
By Theorem \ref{nart}, $\mathcal Q = \add \GG$ has $n$-almost split sequences.
Especially, we have the following result.

\begin{pro}\label{nalmlambda} Let $\Lambda$ be a Koszul $n$-translation algebra and let $\GG$ be the dual $n$-translation algebra of $\Lambda$. Then $\mathrm{add}\,\Gamma$ has $n$-almost split sequences.\end{pro}

Let $\Lambda$ be an $n$-translation algebra with bound quiver $Q=(Q_0,Q_1, \rho)$ and $n$-translation $\tau$.
Note that the complex \eqref{sourseseq} is in fact isomorphic to the Koszul complex \eqref{KoszulComplexii} and $ D_*(\GG^!_t) \simeq  \mathcal E (\GG)_t \simeq \LL_t.$
Thus $P^{t}(S(i)) \simeq \GG \otimes_{\GG_0} \LL_t e_i$, thus if $\LL_t e_i \simeq  \bigoplus_j S(j)^{r_j}$, then $$P^{t}(S(i)) \simeq \bigoplus_j (\GG e_j) ^{r_j}.$$
Such information are easily read out from the $n$-translation quiver using {\em $\tau$-hammock starting at $i$,} which is defined for each $i\in Q_0$ as the quiver  $H^i$ with the vertex set $$H^i_0=\{ (j,t) | j \in Q_0, \mbox{ there is a bound path of length $t$ from $i$ to $j$} \}, $$
and the arrow set
$$\arr{l}{ H^i_1=\{ (\alpha , t): (j,t) \longrightarrow (j', t+1)|\alpha: j\to j'\in Q_1, \mbox{there is bound path $p$} \\ \qquad\qquad \mbox{of length $t$ from $i$ to $j$ such that $\alpha p$ is a bound path of length $t+1$} \}.} $$\normalsize
Define the hammock functions $\mu^i: H^i_0 \longrightarrow \mathbb Z$ to be the integral  maps on the vertices of $H^i$ as follows.
For each $(j,t)\in H^i_0$, define $\mu^i(j,t)$ to be the number of the bound paths of length $t$ from $i$ to $j$ in a maximal linearly independent set.
Let $H^i$ be the hammock starting at $i$ with hammock function $\mu^i$, write $H^i(t)=\{(j,t)\in H^i_0\}$.
Let $i \in Q_0\setminus \mathcal I$, the $\tau$-hammock $H^i$ has the property  that $H(n+1)= \{(\tau^{-1} i, n+1)\}$ and $\mu_i(\tau^{-1} i, n+1 )=1$.

We clearly  have the following result describing the radical series of indecomposable projective $\Lambda$-modules.
\begin{pro}\label{radsoc}
Let $\Lambda$ be an $n$-translation algebra with bound quiver $Q=(Q_0,Q_1, \rho)$ and $n$-translation $\tau$, then
$$\Lambda_t e_i \simeq   \bigoplus_{(j,t)\in H^i_0} S(j)^{\mu^i(j,t)}.$$
\end{pro}

We have the following theorem describing the $n$-almost split sequences in $\mathcal Q=\add \cE(\LL)$, using Koszul complexes by Propositions \ref{nalmost}, and Proposition \ref{radsoc}.

\begin{thm}\label{nalterm}
Assume that $\LL$ is an $n$-translation algebra with $n$-translation quiver $Q$ and $n$-translation $\tau$.
Let $\GG=\mathcal E(\LL)$ and let $\cQ=\add \GG$ be the category of finitely generated graded projective $\Gamma$-modules.
If $$M^{n+1} \to M^{n} \to \cdots \to M^0$$ is an $n$-almost split sequence in $\mathcal Q$ and $M^0$ is indecomposable.
Then there is $i\in Q_0 \setminus \cP$, such that $M^0 = \Gamma e_i$, $$M^t = \bigoplus_{(j,t)\in H^i_{0}}  (\Gamma e_j)^{\mu_i(j,t)} $$ for $0\le t\le n$, and $M^n \simeq \Gamma e_{\tau^{-1} i} $.
\end{thm}

\section{Examples}
The following examples show how $n$-translation algebra are applied to the path algebra of Dynkin type.
In \cite{gl15}, we present some classes of $n$-translation algebras related to higher representation theory \cite{i4,hio}.

\begin{exa}
{\em
In \cite{i4}, Iyama characterized absolute  $n$-complete algebra $T^{m}_{n}(k)$ with dominant dimension no less than $n$ and global dimension no more than $n$.
We show that such algebras can be obtained as truncations from the Skew group algebras of cyclic groups over the polynomial algebras \cite{g13}.

Using results in this paper, we point out that the Auslander Reiten quiver of  the cone $\mathcal M$ of Iyama's absolute $n$-complete algebra $T^{4}_{n}(k)$ can also be obtained as a truncation from the quiver of the $0$-extension of an  $n$-translation algebra for $n=1, 2, \cdots $ (For the detail and for general case see \cite{gl15}).

Starting with quiver $ Q(1)=Q:\stackrel{1}{\circ}\longrightarrow \stackrel{2}{\circ} \longrightarrow \stackrel{3}{\circ}\longrightarrow \stackrel{4}{\circ}$. Let $\Lambda(1)$ be the $0$-translation algebra with quiver $Q(1)$ and relations all the paths of length $2$.
Then $$T^{4}_{1}(k)\simeq \GG(1)=\cE(\LL(1))$$ is just the quadratic dual $\LL(1)$.

It follows from \cite{bbk}, that the trivial extension $\tL(1)$ of $\LL(1)$ is $(2,m-1)$-Koszul, hence $\LL(1)$ is extendable, and its quiver is $\tilde{Q}(1):\xymatrix@C=0.8cm@R0.3cm{\stackrel{1}{\circ}\ar@/^/[r]^{\alpha_1}& \stackrel{2}{\circ} \ar@/^/[l]^{\alpha^*_1}\ar@/^/[r]&{3}\ar@/^/[l] \ar@/^/[r]^{\alpha_{m-1}}&\stackrel{4}{\circ}\ar@/^/[l]^{\alpha^*_{m-1}
}}$

Now we have a $1$-translation algebra $\olL(1) = \tL(1)\#\mathbb Z$ with $1$-translation quiver $\mathbb Z|_1Q(1) =\mathbb ZQ(1)$.\\
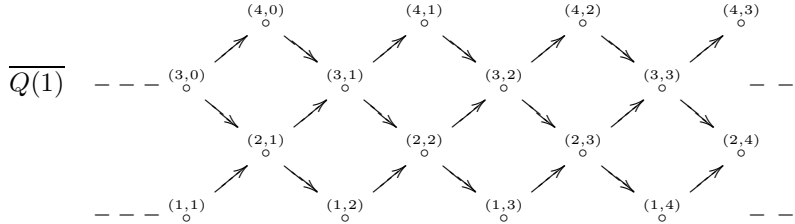
\begin{tabular}{cr}
{$\arr{c}{\\ \\ \\ \\ \overline{Q(1)}\\}$}{\tiny$${\xymatrix@C=0.3cm@R0.3cm{
&&&\stackrel{(4,0)}{\circ}\ar[dr]&&\stackrel{(4,1)}{\circ}\ar[dr]&& \stackrel{(4,2)}{\circ}\ar[dr]&&\stackrel{(4,3)}{\circ}&&&\\
&&\stackrel{(3,0)}{\circ}\ar@{--}[ll]\ar[ur]\ar[dr]&&\stackrel{(3,1)}{\circ}\ar[ur]\ar[dr]&& \stackrel{(3,2)}{\circ}\ar[ur]\ar[dr]&&\stackrel{(3,3)}{\circ}\ar[ur]\ar[dr]&\ar@{--}[r]&&&\\
&&&\stackrel{(2,1)}{\circ}\ar[ur]\ar[dr] && \stackrel{(2,2)}{\circ}\ar[ur] \ar[dr] &&\stackrel{(2,3)}{\circ}\ar[ur] \ar[dr] &&\stackrel{(2,4)}{\circ}\\
&&\stackrel{(1,1)}{\circ}\ar@{--}[ll]\ar[ur]&& \stackrel{(1,2)}{\circ}\ar[ur]&& \stackrel{(1,3)}{\circ}\ar[ur]&& \stackrel{(1,4)}{\circ}\ar[ur]&\ar@{--}[r]&
}}$$}\normalsize
\end{tabular}\\
In this quiver $Q(1) \simeq Q(1)\times \{1\}$ is a $\tau_{[1]}$-slice.
Now for each vertex $\bar{i} =(i,1)$, consider the hammock $H^{\bar{i}}$ starting at $\bar{i}$, which is used in \cite{g02} to describe the linear part of the projective resolution of $S(\bar{i})$.
Let $Q(2)$  be the truncation using the hammocks starting at  the vertices on $ Q(1)\times \{1\}$.\\
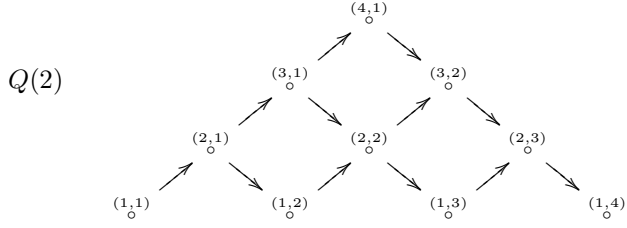
\begin{tabular}{cr}
{$\arr{c}{\\ \\ \\ \\ Q(2) \\}$} &{\tiny$${\xymatrix@C=0.3cm@R0.3cm{
&&&\stackrel{(4,1)}{\circ}\ar[dr]&& &&&&&\\
&&\stackrel{(3,1)}{\circ}\ar[ur]\ar[dr]&& \stackrel{(3,2)}{\circ}\ar[dr]&&&&&&\\
&\stackrel{(2,1)}{\circ}\ar[ur]\ar[dr] && \stackrel{(2,2)}{\circ}\ar[ur] \ar[dr] &&\stackrel{(2,3)}{\circ} \ar[dr] &&
\\
\stackrel{(1,1)}{\circ}\ar[ur]&& \stackrel{(1,2)}{\circ}\ar[ur]&&\stackrel{(1,3)}{\circ}\ar[ur]&&\stackrel{(1,4)}{\circ}&&
}}$$}\normalsize
\end{tabular}\\
We have a $1$-translation algebra $\Lambda(2)$ defined on $Q(2)$.
We remark that though $\add \cE(\tL(1)$ and  $\add \cE(\olL(1)$ have not $1$-almost split sequence,  $\add \cE(\LL(2))$ do have $1$-almost split sequences, by Theorem \ref{nart}.
In fact, this category is equivalent to the cone of $T^{4}_{1}(k)$, and we have that $ T^{4}_{2}(k) \simeq \GG(2)= \cE(\LL(2))$, and its quiver is a truncation of $\mathbb ZQ(1)=\mathbb Z|_0Q(1)$.
}\end{exa}

The general case for arbitrary $m$ and $n$ are studied in \cite{gl15}.
It can be regard as a version of the classical result that the Auslander Reiten quiver of the path algebra $kQ$ for a Dynkin quiver $Q$ is a truncation of $\mathbb ZQ$.

\begin{exa}{\em
Let $\GG=kQ$ be the path algebra of quiver $Q$, and let $\LL =\mathcal E (\GG)$ be its quadratic dual.
Then $\LL$ is a algebra with radical squared zero.
In general, $\LL$ is not a $0$-translation algebra, but the trivial extension $\tL$ is a $1$-translation algebra by Theorem 2.1 of \cite{bbk}.
Thus $\tL\ltimes \mathbb Z^*$ is a self-injective $1$-translation algebra with quiver $\mathbb Z Q$ as its translation quiver.
If $Q$ is a Dynkin quiver with Coxeter number $h$, $\olG= \cE(\tL\ltimes \mathbb Z^*)$ is an self-injective $h-3$-translation algebra by Theorem 2.1 of \cite{bbk} and Theorem \ref{eqntrans}.
Now take $Q(1)$ as the full subquiver  of $\mathbb Z Q$ with the vertex set $$
\{ (j,t) |((j,t),l)\in H_0^{i}, i\in Q_0, 0\le l \le h-2\}.$$
Let $\GG(2)$ be the quotient of $\olG$ with $Q(2)$ and $\LL(2)=\cE(\GG(2))$.
The $\LL(2)$ is $1$-translation algebra.
$\add \olG$ and $\add \cE (\tL)$ have not ($1$-)almost split sequence.
But $\add\GG(2)$ has $1$-almost split sequences.

We remark that $\cE (\tL)$ is the preprojective algebra of $\GG$ and $\GG(1)$ is the Auslander algebra of $\GG$.

If $Q$ is not Dynkin, then $\tL$ and $\olL$ are both Koszul, so both $\add\cE(\tL)$ and
$\add \cE(\tL\ltimes \mathbb Z^*)$ have $1$-almost split sequences.

}\end{exa}

\section*{Acknowledgements}
The author would like to thank the referee for suggestions and comments on revising and improving the paper.

{}
\end{document}